\documentclass[11pt]{article}

\usepackage[T1]{fontenc}
\usepackage[latin1]{inputenc}
\usepackage{epsfig}
\usepackage{amsmath,amssymb}
\usepackage[active]{srcltx}
\usepackage{dsfont}
\usepackage[titles]{tocloft}
\usepackage[usenames,dvipsnames]{color}
\usepackage{hyperref}
\usepackage{tikz}
\usetikzlibrary{calc}
\usepackage{MnSymbol}

\definecolor{jaune}{RGB}{0,0,0}

\definecolor{violet2}{RGB}{0,0,0}
\definecolor{vert}{RGB}{0,0,0}
\definecolor{rouge}{RGB}{0,0,0}
\definecolor{bleu}{RGB}{0,0,0}
\definecolor{orange2}{RGB}{0,0,0}
\definecolor{violet}{RGB}{0,0,0}
\definecolor{menthe}{RGB}{0,0,0}

\numberwithin{equation}{section}
\newtheorem{theoreme}{Theorem}[section]

\newtheorem{proposition}[theoreme]{Proposition}
\newtheorem{corollary}[theoreme]{Corollary}
\newtheorem{remarque}[theoreme]{Remark}

\newtheorem{lemme}[theoreme]{Lemma}

\newtheorem{definition}[theoreme]{Definition}
\newtheorem{notation}[theoreme]{Notation}

\newenvironment{proof}[1][Proof]{\noindent \textbf{#1.}~ }
{\hfill\rule{2mm}{2mm} \vspace{\parskip} }

\newcommand{\RR}{\ensuremath{\mathbb R}}

\newcommand{\PP}{\ensuremath{\mathbb P}}

\newcommand{\ZZ}{\ensuremath{\mathbb Z}}
\newcommand{\EE}{\ensuremath{\mathbb E}}
\newcommand{\NN}{\ensuremath{\mathbb N}}

\newcommand{\CL}{\ensuremath{\mathcal L}}

\newcommand{\T}{\ensuremath{\mathcal T}}

\newcommand{\F}{\ensuremath{\mathcal F}}

\newcommand{\G}{\ensuremath{\mathcal G}}

\newcommand{\xdashrightarrow}[2][]{\ext@arrow 0359\rightarrowfill@@{#1}{#2}}

\newcommand{\wT}{\ensuremath{\widetilde{\mathcal T}}}
\newcommand{\wSigma}{\ensuremath{\widetilde{\Sigma}}}
\newcommand{\wtau}{\ensuremath{\widetilde{\tau}}}

\newcommand{\A}{\ensuremath{\mathcal{A}}}

\newcommand{\1}{\mathds{1}}

\newcommand{\De}{\Delta}

\newcommand{\de}{\delta}

\newcommand{\si}{\sigma}

\usetikzlibrary{decorations.pathmorphing}
\tikzset{snake it/.style={decorate, decoration=snake}}

\title{Recursive games: Uniform value, Tauberian theorem and the Mertens conjecture ``$Maxmin=\lim v_n=\lim v_\lambda$''}
\author{Xiaoxi LI \thanks{CNRS, IMJ-PRG, UMR 7586, Sorbonne Universit\'es, UPMC Univ. Paris 06, Univ. Paris Diderot, Sorbonne Paris Cit\'e, Paris, France. Email: xxleewhu@gmail.com.},
\  Xavier VENEL \thanks{CES, Universit\'e Paris 1 Panth\'eon Sorbonne, Paris. France. Email: xavier.venel@univ-paris1.fr}}

\date{May 30, 2015}

\begin{document}
\maketitle
\tableofcontents

\newpage

\begin{abstract} 
We study two-player zero-sum recursive games with a countable state space and finite action spaces at each state. When the family of $n$-stage values $\{v_n,n\geq 1\}$ is totally bounded for the uniform norm, we prove the existence of the uniform value. Together with a result in Rosenberg and Vieille \cite{Rosenberg_2000}, we obtain a uniform Tauberian theorem for recursive game: $(v_n)$ converges uniformly if and only if $(v_\lambda)$ converges uniformly.

We apply our main result to finite recursive games with signals (where players observe only signals on the state and on past actions). When the maximizer is more informed than the minimizer, we prove the Mertens conjecture $Maxmin=\lim_{n\to\infty} v_n=\lim_{\lambda\to0}v_\lambda$. Finally, we deduce the existence of the uniform value in finite recursive game with symmetric information.
\end{abstract}

 

\section {Introduction}

Stochastic games were introduced by Shapley \cite{Shapley_53} to model a multiplayer dynamic interaction, where players' collective decisions influence the current payoff and also the future state. In this article, we focus on two-player zero-sum recursive games introduced by Everett \cite{Everett_57}. The specificity of a recursive game is that the state space is divided into two sets: absorbing states and active states. On absorbing states, the process is absorbed and the payoff is fixed. On active (non-absorbing) states, the payoff is always equal to $0$.

There are several ways to evaluate the payoff stream in a zero-sum stochastic game. Given a positive integer $n$, the $n$-stage payoff is the expected average payoff during the first $n$ stages. Given $\lambda\in(0,1]$, the $\lambda$-discounted payoff is the Abel mean of the infinite stage payoffs with a weight $\lambda(1-\lambda)^{t-1}$ for stage $t$. We will focus on the concept of uniform value. A stochastic game admits a uniform value if both players can approximately guarantee the same payoff level in all sufficiently long $n$-stage games without knowing \textit{a priori} the length of the game.  

Mertens and Neyman \cite{Mertens_81} proved that a stochastic game with a finite state space {\color{violet2} and finite set of actions} where the players observe the current state and the stage payoffs admits a uniform value. Their proof uses the fact that  the function $\lambda\mapsto v_\lambda$ has bounded variation, where $v_\lambda$ is the $\lambda$-discounted value (Bewley and Kohlberg \cite{Bewley_76a}). For stochastic games with an infinite state space, this argument in general does not apply.

Markovian decision processes (henceforth MDP) are stochastic games with only one player. Lehrer and Sorin \cite{Lehrer_92} showed that in a MDP, the uniform convergence of $(v_\lambda)$ (w.r.t. the initial state) as $\lambda$ tends to zero is equivalent to the uniform convergence of the $n$-stage values $(v_n)$ as $n$ tends to infinity. Nevertheless, uniform convergence of $(v_n)$ or $(v_\lambda)$ is not sufficient for the existence of the uniform value (cf. Monderer and Sorin \cite{Monderer_93} {\color{vert} or} Lehrer and Monderer \cite{Lehrer_94}).

For recursive games, the situation seems to be different. There are two results giving sufficient conditions for {\color{vert}a} recursive game with countable state space to have a uniform value. The first one can be derived from Rosenberg and Vieille \cite{Rosenberg_2000}: if {\color{vert} $(v_\lambda)$  converges uniformly to some function $v$, then the recursive game has a uniform value, which is equal to $v$. The second one is due to Solan and Vieille \cite{Solan_2002}: if, except on a finite subset, the \textit{limsup value}\footnote{The limsup value is the value of the game in which the global payoff to player 1 is the limsup of the stage payoff stream.} is above a strictly positive constant on the non-absorbing states, then the recursive game has a uniform value, which is equal to the limsup value.\\

\noindent The main result of this paper is that the uniform convergence of the $n$-stage values is a sufficient condition for the existence of the uniform value. In fact we prove a stronger result: for any recursive game with countable state space, if the family {\color{vert} $\{v_n,n\geq 1\}$} is totally bounded for the uniform norm, then the {\color{vert} uniform value} exists.  {\color{jaune} Our proof follows the same idea as Solan and Vieille \cite{Solan_2002} and we will use several of their results}.

Our result together with the result of Rosenberg and Vieille \cite{Rosenberg_2000} provides a uniform Tauberian theorem for recursive games: $(v_n)$ converges uniformly if and only if $(v_\lambda)$ converges uniformly, and in case of convergence, both limits are the same. \textcolor{rouge}{For general stochastic games, Ziliotto \cite{Ziliotto_2015} provided recently a direct proof \textcolor{violet2}{of this result}.}

Finally, we apply our main result to finite {\color{orange2} {\color{rouge} recursive games with signals}}. In a {\color{orange2} {\color{rouge} recursive game with signals}}, players do not perfectly observe the state and actions at every stage anymore, rather they receive a private signal. Mertens \cite{Mertens_87} conjectured that in a general model of zero-sum repeated games, if player 1 (the maximizer) is always more informed than player 2 (the minimizer) during the play (in the sense that player 2's private signal can be deduced from player 1's private signal) then $Maxmin=\lim_{n\to\infty}v_n=\lim_{\lambda\to 0}v_\lambda$, $i.e.$, both the uniform maxmin and the asymptotic value exist and are equal. 

{\color{rouge} Ziliotto \cite{Ziliotto_2013} showed that the result is false in general. Nevertheless, several positive results have been obtained for subclasses of games including} Sorin \cite{Sorin_84} and Sorin \cite{Sorin_85b} for Big match with one-sided incomplete information, Rosenberg \textit{et al}. \cite{Rosenberg_2004}, Renault \cite{Renault_2012a} and Gensbittel \textit{et al}. \cite{Gensbittel_2014} for a more informed controller, and Rosenberg and Vieille \cite{Rosenberg_2000} for recursive games with one-sided incomplete information. 

We prove the { Mertens conjecture in finite {\color{orange2} {\color{rouge} recursive games with signals}}, where player 1 is always more informed than player $2$ during the play. The proof uses several results from Gensbittel \textit{et al}. \cite{Gensbittel_2014}, concerning the $n$-stage value functions in a repeated game where player 1 is more informed than player 2. Our result generalizes Rosenberg and Vieille \cite{Rosenberg_2000}, which deals with the model where player 1 is informed of a private signal on the state at the beginning of the game. {\color{vert} Moreover, we deduce the existence of the uniform value in finite recursive games with symmetric information. } 
 
The organization of the article is as follows: in Section \ref{sec:2} we introduce the model of recursive games; in Section \ref{sec:3} we present the main result and several corollaries; Section \ref{sec:4} is dedicated to the proofs; finally in Section \ref{sec:5} we apply the result to finite {\color{orange2} {\color{rouge} recursive games with signals}}.\\ 

\section{Preliminaries: model and notations} \label{sec:2}

\textbf{Notation} \  Given any metric space $S$, endowed with the Borelian $\sigma$-algebra, we denote by $\Delta(S)$ the set of probabilities on $S$ and we denote by $\Delta_f(S)$ the set of probabilities with finite support.

\subsection{The model}


A two-player zero sum stochastic game $\Gamma=\langle X,A,B,g,q\rangle$ is given by
\begin{itemize}
\item a state space $X$.
\item player 1's action set $A$, and for any $x\in X$, $A(x)$ is a finite subset of $A$.
\item player 2's action set $B$, and for any $x\in X$, $B(x)$ is a finite subset of $B$.
\item a  payoff function: $g: X\times A\times B \to [-1,+1]$.
\item a transition probability function: $q$: $X \times A \times B\to\Delta_f(X)$. 
\end{itemize}

\noindent \textbf{Play of the game}  \ The stochastic game with initial state $x_1\in X$ is denoted by $\Gamma(x_1)$, and {\color{vert} is} played as follows: 
at each stage
$t\geq 1$, after observing $(x_1,a_1,b_1,...$ $...,a_{t-1},b_{t-1},x_t)$, player $1$ and player $2$ choose simultaneously
actions $a_t \in A(x_t)$  and $b_t \in B(x_t)$. The stage payoff is $g(x_t,a_t,b_t)$ and a new state $x_{t+1}$ is drawn
according to the probability distribution $q(x_t,a_t,b_t)$. Both players observe the action pair $(a_t,b_t)$ and the state
$x_{t+1}$. The game {\color{vert} then} proceeds to stage $t+1$.\\ 

\noindent Note that we did not make any measurability assumption on the model. As the transition probability distribution is supposed to be finitely supported, given an initial state, the set of actions and states that might appear in the infinite game are in fact countable. Therefore probability distributions are well defined.   \\

\noindent \textbf{Recursive game} \ $\Gamma$ is a recursive game if there {\color{bleu} exist a set of active states denoted by $X^0$ and a set of absorbing states denoted by $X^*$ with $X^0\cup X^*=X$ and $X^0\cap X^*=\emptyset$,} such that:
\begin{itemize}
\item the stage payoff is $0$ on active states: $\forall x \in X^0$, $g(x,a,b)=0, \forall (a,b)\in A(x)\times B(x)$;
\item states in $X^*$ are absorbing: $\forall x\in X^*$, $q(x,a,b)(x)=1$, $\forall (a,b)\in A(x)\times B(x)$, and $g(x,a,b)$ depends only on $x$.
\end{itemize}


\subsection{Definition of strategies and evaluations}

\textbf{History}\ \  At stage $t$, the space of finite histories is
$H_{t}=(X \times A \times B)^{t-1}\times X$. Set
$H_{\infty}=(X\times A \times B)^{\infty}$  to be the space of infinite
\emph{plays}. We consider the discrete topology on $X$, $A$ and $B$. For every $t\geq 1$, we identify any $h_t\in H_t$ with a cylinder set in $H_\infty$ and denote by $\mathcal{H}_t$ the $\sigma$-field of $H_t$ induced on $H_\infty$. The product $\sigma$-field on $H_\infty$ is $\mathcal{H}_\infty=\sigma(\mathcal{H}_t,t\geq 1)$.\\
 
\noindent \textbf{Strategy}\ \  A \emph{(behavior) strategy} for player $1$ is a sequence of functions $\sigma=(\sigma_t)_{t
\geq 1}$ with each $t\geq 1$, $\sigma_t: (H_t, \mathcal{H}_t) \rightarrow \Delta(A)$ such that for every $h_t\in H_t$, $\sigma_t(h_t)(A(x_t))=1$. If for every $t\geq 1$ and $h_t \in H_t$, there exists $a\in A(x_t)$ such that $\sigma_t(h_t)[a]=1$, then the strategy is \emph{pure}. {\color{bleu} We define similarly} a behavior strategy $\tau$ for player 2. Denote by $\Sigma$ and $\T$ respectively player 1's and player 2's sets of behavior strategies. {\color{jaune} Denote by $\widehat{\Sigma}$ and $\widehat{\T}$ respectively player 1's and player 2's subsets of strategies that depend on the histories only through the states but not on the actions.}\\
 
\noindent \textbf{Evaluations}\   \  Let us describe several ways to evaluate the payoff in $\Gamma$. By Kolmogorov's extension theorem, any triple $(x_1, \sigma,\tau)\in X\times \Sigma\times \T$ induces a unique probability distribution over $(H_\infty, \mathcal{H}_\infty)$ denoted by $\PP_{x_1,\sigma,\tau}$. Let $\EE_{x_1,\sigma,\tau}$ be the corresponding expectation. \\

\noindent \textit{n-stage average} \ For each positive $n\geq 1$, the expected average payoff up to stage $n$, induced by the couple of strategies $(\sigma,\tau)$ and the initial state $x_1$ is given by
\[
\gamma_n(x_1,\sigma,\tau)=\EE_{x_1,\sigma,\tau}\left(\frac{1}{n} \sum_{t=1}^n g(x_t,a_t,b_t)\right).
\]
The game with expected $n$-stage average payoff and initial state $x_1$ is denoted as $\Gamma_n(x_1)$.\\

\noindent \textit{$\lambda$-discounted average}  \ For each  $\lambda \in (0,1]$, the expected $\lambda$-discounted average payoff, induced by the couple of strategies $(\sigma,\tau)$ and the initial state $x_1$ is given by
\[
\gamma_\lambda(x_1,\sigma,\tau)=\EE_{x_1,\sigma,\tau}\left(\lambda
\sum_{t=1}^\infty (1-\lambda)^{(t-1)} g(x_t,a_t,b_t)\right).
\]
The game with expected $\lambda$-discounted average payoff and initial state $x_1$ is denoted as $\Gamma_\lambda(x_1).$\\

\noindent In either $\Gamma_n(x_1)$ or $\Gamma_\lambda(x_1)$, player 1 maximizes the expected average payoff and player 2 minimizes it. For a fixed $x_1$ the game $\Gamma_n(x_1)$ is finite, so there exists a value $v_n(x_1)$ by {\color{vert} minmax theorem}. The existence of the discounted value $v_\lambda(x_1)$ is also standard, and we refer to Mertens {\it et al.} \cite{Mertens_2015} (Section VII.1.) for a general presentation.

\subsection{Stopping time and concatenation of strategies}
 

%

A function $\theta:(H_\infty,\mathcal{H}_\infty) \rightarrow \NN$ is called a \emph{stopping time} if the set $\{h \in H_\infty| \theta(h) =t \}$ is  {$\mathcal{H}_t$-measurable} for all $t\geq 1$. Explicitly for any $h,h'\in H_\infty$ and $n\geq 1$: if $h$ and $h'$ coincide until stage $n$ and $\theta(h)= n$ then $\theta(h')=n$. Let $\theta$ and $\theta'$ be two stopping times, we write $\theta \leq \theta'$ if for every $h\in H_\infty$, $\theta(h) \leq \theta'(h)$. 

Given a sequence of strategies $(\sigma^{[\ell]})_{\ell\geq 1}$ and a sequence of increasing stopping time $(\theta_\ell)_{\ell\geq 1}$, we define $\sigma^*:=\sigma^{[1]}\theta_1\sigma^{[2]}\theta_2\cdot\cdot\cdot$ as the \emph{concatenation} of $(\sigma^{[\ell]})_{\ell\geq 1}$ along $(\theta_\ell)_{\ell\geq 1}$. {\color{orange2} Given $n\geq t\geq 1$ and $h\in H_\infty$, let $h_n$ be the projection of $h$ on $H_n$ and $h_n^t$ be the history of $h$ between stage $t$ and $n$. The strategy $\sigma^*$ is defined by $\sigma^*_n\big(h_n\big)=\sigma_n^{[1]}(h_n)$ if $n<\theta_1(h)$; $\sigma^*_n(h_n)=\sigma_{n-\theta_{m-1}}^{[m]}(h_n^{\theta_{m-1}})$ if $\theta_{m-1}\leq n< \theta_m$. Informally, for every {\color{rouge} $\ell\geq 1$} at stage $\theta_\ell$, the player forgets the past and starts to play $\sigma_{\ell+1}$ at the current state.}

\subsection{Uniform value}

\noindent \textbf{Uniformly guarantee}\  Player 1 \textit{uniformly guarantees} $w$ if for every $\varepsilon>0$, there exists $\sigma_\varepsilon$ in $\Sigma$ and $N_0\geq 1$ such that for every $x_1\in X^0$,
\begin{align*}
 \gamma_n(x_1,\sigma_\varepsilon,\tau) \geq w(x_1)-\varepsilon, \ \  \forall n\geq N_0, \ \forall \tau\in \T.
\end{align*} 
We say that the strategy $\sigma_\varepsilon$ uniformly guarantees $w-\varepsilon$.  Similarly, player 2 uniformly guarantees $w$ if for every $\varepsilon>0$, there exists $\tau_\varepsilon$ in $\T$ and $N_0\geq 1$ such that for every $x_1\in X^0$,
 \begin{align*}
 \gamma_n(x_1,\sigma,\tau_\varepsilon) \leq w(x_1)+\varepsilon, \ \ \forall n\geq N_0, \ \forall \sigma\in \Sigma.
\end{align*}

\noindent \textbf{Uniform value}\ $v_\infty:X\to\RR$ is the uniform value of the game $\Gamma$ if both players uniformly guarantee $v_\infty$. A strategy for player 1 ($resp.$ player 2) that uniformly guarantees $v_\infty-\varepsilon$ ($resp.$ \  $v_\infty+\varepsilon$) is called \emph{uniform $\varepsilon$-optimal}. If both players can uniformly guarantee $v_\infty$ with pure strategies, $\Gamma$ has \emph{a uniform value in pure strategies}. 

\begin{remarque}  \label{rem:uniformcon} In defining the uniform value, we ask $N_0$ to be independent of the initial state $x_1$. One direct consequence of the existence of the uniform value $v_\infty$ is the uniform convergence of $(v_n)_{n\geq 1}$ to $v_\infty$. This is stronger than the definition where  the existence of the uniform value is considered state by state (see for example Solan and Vieille \cite{Solan_2002}, Definitions 3-4)

\end{remarque}
 
%
 
\section{Main results} \label{sec:3}
In this section, we present the main result of the paper, namely Theorem \ref{theo1}, as well as several corollaries. We also provide an example that does not satisfy the condition of Theorem \ref{theo1} and {\color{vert}{does not have}} a uniform value.

\subsection{Sufficient condition for the existence of the uniform value}

Denote by ${\bf B}(X)$ the set of functions from $X$ to $[-1,1]$ with the uniform norm $\|\cdot\|_\infty$. Recall that a set of functions $F$ in $({\bf B}(X),\|.\|_\infty)$ is \emph{totally bounded} if for every $\varepsilon>0$, there exists a finite subset $F_R=\{f_r: 1\leq r\leq R\} \subseteq F$ such that for any $f\in F$, there is $f_r\in F_R$ with $||f-f_r||_\infty\leq \varepsilon.$
\begin{theoreme}\label{theo1}
Suppose that the space $\{v_n, n\geq 1\}$ is totally bounded for the uniform norm, then the recursive game $\Gamma$ has a uniform value $v_\infty$. Moreover both players can uniformly guarantee $v_\infty$ with strategies that depend only on the history of states and not on past actions.
\end{theoreme}

We deduce from \textcolor{bleu}{the previous result} a uniform Tauberian theorem in recursive games.
\begin{corollary} \label{coro:UCV_lambda}
The sequence of $n$-stage values $(v_n)_{n\geq 1}$ converges uniformly as $n$ tends to infinity if and only if the sequence of $\lambda$-discounted values $(v_{\lambda})_{\lambda \in (0,1]}$ converges uniformly as $\lambda$ tends to zero. In case of convergence, both limits are the same.
\end{corollary}
On one hand, if $(v_n)$ converges uniformly, the family is totally bounded, thus the uniform value exists, and this implies the uniform convergence of $(v_\lambda)$ (Sorin \cite{Sorin_2002}, Lemma 3.1). On the other hand, the converse result is established  in Rosenberg and Vieille \cite{Rosenberg_2000} (see Remark 6, Theorem 1 and Theorem 3). 
 
\begin{remarque} The equivalence of the uniform convergences of $(v_n)_{n\geq 1}$ and $(v_\lambda)_{\lambda\in(0,1]}$ has been proven in MDP by Lehrer and Sorin \cite{Lehrer_92}. {\color{orange2} Ziliotto \cite{Ziliotto_2015} recently showed that it is also true for stochastic games whenever the Shapley operator is well defined.}
\end{remarque}
 
If, in addition, for every $n\geq 1$ the $n$-stage value $v_n(x)$ exists in pure strategies, then $\Gamma$ has a uniform value in pure strategies.
 
\begin{corollary}\label{theo3}
Suppose that for every $n\geq 1$, both players have pure optimal strategies in the $n$-stage game, and $\{v_n , n\geq 1\}$ is totally bounded for the uniform norm. Then $\Gamma$ has a uniform value $v_\infty$ in pure strategies. Moreover, both players can uniformly guarantee $v_\infty$ with strategies that depend only on the history of states and not on past actions.
\end{corollary}

\begin{remarque}\label{theo3_com}
The result in Corollary \ref{theo3} extends to games with general action sets $A(x)$ and $B(x)$ provided that for any $n\geq 1$, the $n$-stage game has a value and both players have pure optimal strategies{\color{rouge}.}
\end{remarque}

The proof of Corollary \ref{theo3} is similar to that of Theorem \ref{theo1}. The key difference involves a technical lemma (Lemma \ref{optstop}) for the existence of a (pure) stopping time which is used in the definition of players' optimal strategies (\textcolor{bleu}{see} the proof of Proposition \ref{posstage}). We discuss this point and present the proof in Subsection \ref{pure_proof}.

\subsection{A recursive game without uniform value}

We present here an example of a recursive game with countable state space where $\{v_n, \ n\geq 1 \}$ is not totally bounded and there is no uniform value  {\color{orange2} (See \textbf{Figure 3.2} below for illustration)}. {\color{rouge} This is an adaptation to our framework of an example in Lehrer and Sorin \cite{Lehrer_92}.}\\
 
{\color{bleu} The state space is a subset of $\ZZ\times\ZZ$. The set of active states is $X^0=\{(x,y)\in\NN\times\NN\ | 0\leq y\leq x\}$ and the set of absorbing states is $X^*=X^*_1\bigcup X^*_{-2}$ (two types), where $X^*_1=\NN \times \{-1\}$ and $X^*_{-2}=\{(x,x+1)| x \geq 0 \}$. The payoff is $1$ on $X^*_1$ and is $-2$ on $X^*_{-2}$.} There is only one player (maximizer), whose action set is $\{R(ight), J(ump)\}$. The transition rule is given by:
\begin{itemize}
\item at $(x,0)\in X^0$: $q\big((x,0), R\big)(x+1,0)=1$, and $q\big((x,0), J\big)(x,-1)=q\big((x,0), J\big)(x,1)=\frac{1}{2}$;
\item at $(x,y)\in X^0$ with $0<y\leq x$: $q\big((x,y),a \big)(x,y+1)=1$, $\forall a\in\{R,J\}$.
\end{itemize}

Starting at $(0,0)$, one optimal strategy for an $n$-stage game is to go \textit{Right} for half of the game, and then to \textit{Jump}. This gives an expected average payoff around $\frac{1}{4}$, thus $\lim_{n\to\infty} v_n(0,0)=\frac{1}{4}$. In a $\lambda$-discounted game, the optimal stage to $Jump$ is approximately $\frac{\ln (\frac{2-\lambda}{4})}{\ln (1-\lambda)}$. It follows that $v_{\lambda}(0,0)\approx \frac{2-\lambda}{16}$ and thus $\lim_{\lambda\to 0} v_{\lambda}(0,0)=\frac{1}{8}$. This implies that there is no uniform value. On the other hand, $\{v_n, n\geq 1\}$ is not totally bounded for the uniform norm. Indeed, the convergence of $(v_n)$ is not uniform: for any $x\geq 1$, $\lim_{n\to\infty} v_n(x,1)=-2$ while $v_x(x,1)=0$. 

 
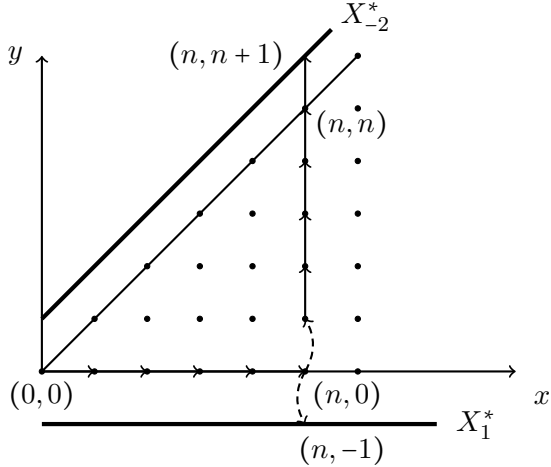
\begin{center}
\begin{tikzpicture}[scale=0.7] 
\centering
\draw [thick, <->] (9,0) -- (0,0) -- (0,6); 
\draw[thick, domain=0:6] plot (\x, {\x}); 
\draw[fill] (0,0) circle (0.05) (1,0) circle (0.05) (2,0) circle (0.05) (3,0) circle (0.05) (4,0) circle (0.05) (5,0) circle (0.05) (6,0) circle (0.05) (1,1) circle (0.05) (2,1) circle (0.05) (3,1) circle (0.05) (4,1) circle (0.05) (5,1) circle (0.05) (6,1) circle (0.05) (2,2) circle (0.05) (3,2) circle (0.05) (4,2) circle (0.05) (5,2) circle (0.05) (6,2) circle (0.05)  (3,3) circle (0.05) (4,3) circle (0.05) (5,3) circle (0.05) (6,3) circle (0.05) (4,4) circle (0.05) (5,4) circle (0.05) (6,4) circle (0.05) (5,5) circle (0.05) (6,5) circle (0.05) (6,6) circle (0.05);

\node [below] at (0,0) {$(0,0)$};
  
\draw[ultra thick, domain=0:5.5] plot (\x, {1+\x});
\node  at (3.5,6) {$(n, n+1)$};
  
\draw[ultra thick, domain=0:7.5] plot (\x, {-1});
\node  at (5.7,-1.5) {$(n,-1)$};

\node at (9.5, -0.5) {$x$};
\node at (-0.5, 6) {$y$};

\node at (8.25, -1) {$X^*_1$};
\node at (6.15, 6.75) {$X^*_{-2}$};  

\draw  [thick, ->] (0,0) to [right] (1,0);
\draw  [thick, ->] (1,0) to [right] (2,0);
\draw  [thick, ->] (2,0) to [right] (3,0);
\draw  [thick, ->] (3,0) to [right] (4,0);
\draw  [thick, ->] (4,0) to [right] (5,0);
  
\draw  [thick, ->, densely dashed] (5,0) to [bend right] (5,1);
\draw  [thick, ->] (5,1) -- (5,2);
\draw  [thick, ->] (5,2) -- (5,3);
\draw  [thick, ->] (5,3) -- (5,4);
\draw  [thick, ->] (5,4) -- (5,5);
\draw  [thick, ->] (5,5) -- (5,6);
\draw  [thick, ->,  densely dashed] (5,0) to [bend right] (5,-1);

\node [below] at (5.8,0) {$(n,0)$};
\node [right] at (5, 4.7) {$(n,n)$}; 

\node[text width=6cm, anchor=west, right] at (10,3.5)
    {The figure on the left illustrates a play $(R,...,R, J)$ jumping after $n$ steps: with probability 1/2 the state is absorbed at $(n,n+1)\in X^*_{-2}$; with probability 1/2 the state is absorbed at $(n,-1)\in X^*_{1}$.\\
   ~ \\ 
    $\longrightarrow$ : a deterministic transition;\\
     ${-}\dashrightarrow$: a probabilistic transition.};

\end{tikzpicture}

 \textbf{\ \ \  Figure 3.2}
 \end{center}

\section{Proofs} \label{sec:4}
{\color{vert} In the first subsection, we introduce and establish preliminary results for a subclass of recursive game, which will be called \emph{positive-valued recursive games}}. In the second subsection, we prove Theorem \ref{theo1} by a reduction of any recursive game to a positive-valued recursive game. The proof for Corollary \ref{theo3} is given in the third subsection.

\subsection{The case of positive-valued recursive game}\label{positive}
%

\begin{definition} A recursive game is \emph{positive-valued} if there exist $M>0$ and $n_0 \geq 1$ such that for every non-absorbing state $x\in X^0$, there exists $n(x)\leq n_0$ such that $v_{n(x)}(x) \geq M$.
\end{definition}

{\color{rouge} In order to state the next proposition, we first introduce the notion of uniformly terminating strategy.}

\begin{definition}
Denote by $\rho$ the stopping time of absorption in $X^*$: $\rho=\inf \{n \geq 1, x_n \in X^* \}$.  The strategy $\sigma$ is said to be \emph{uniformly terminating} if for any $\varepsilon>0$, there exists $N\geq 1$ such that for every $x_1\in X^0$ and for every $\tau \in \T $, $\PP_{x_1,\sigma,\tau}(\rho \leq N) \geq 1-\varepsilon.$\\
\end{definition}

\begin{proposition}\label{posstage}
Let $\Gamma$ be a positive-valued recursive game. We fix the numbers $M>0, n_0\geq 1$ and the mapping $n(\cdot):X^0\longrightarrow\{1,...,n_0\}$ such that $v_{n(x)}(x)\geq M, \forall x\in X^0$.\\ 
Then player 1 uniformly guarantees $v_{n(\cdot)}(\cdot)$ with uniformly terminating strategies that depends only on states: for all $\varepsilon>0$, there exists $\sigma^*$ in $\widehat{\Sigma}$ and $N_0\geq 1$ such that for every $x_1\in X^0$ and every $\tau$ in $\T$,
$$(i) \  \PP_{x_1, \sigma^*,\tau}(\rho\leq N_0)\geq 1-\varepsilon \text{\  and  \ } (ii) \  \gamma_n(x_1,\sigma^*,\tau)\geq v_{n(x_1)}(x_1)-\varepsilon, \ \forall n\geq N_0.$$
\end{proposition}
\begin{proof} Let $\hat{\sigma}$ be a profile of strategies such that for every $x\in X^0$, $\hat{\sigma}(x)$ is optimal in the $n(x)$-stage game $\Gamma_{n(x)}(x)$. Let $\tilde{k}:=\tilde{k}(x)$ be a random stage uniformly chosen in $\{1,...,n(x)\}$. {\color{violet} For any $\tau\in \T$ and $x\in X^0$, $(x,\hat{\sigma},\tau)$ and $\tilde{k}$ induce a probability distribution over $H_\infty \times \{1,...n(x)\}$, which we denote by $\widetilde{\PP}_{x,\hat{\sigma}, \tau}$. Let $\widetilde{\EE}_{x,\hat{\sigma}, \tau}$ be the corresponding expectation}. We obtain:
$$\widetilde{\EE}_{x,\hat{\sigma}, \tau}[g(x_{\tilde{k}})]=  \EE_{x,\hat{\sigma},\tau}\Big[\frac{1}{n(x)}\sum_{l=1}^{n(x)}g(x_t)\Big] \geq \inf_{\tau'}\EE_{x,\hat{\sigma},\tau'}\Big[\frac{1}{n(x)}\sum_{l=1}^{n(x)}g(x_t)\Big]\geq v_{n(x)}(x)\geq M.$$
{\color{bleu} It follows that
\[
\widetilde{\EE}_{x,\hat{\sigma}, \tau}\left[g(x_{\tilde{k}}) \1_{\rho\leq  \tilde{k}}+ g(x_{\tilde{k}}) \1_{\rho > \tilde{k}}\right]\geq M.
\]
{\color{rouge} On the event $\{\rho>\tilde{k}\}$, $g(x_{\tilde{k}})=0$, whereas on the event $\{ \rho\leq \tilde{k} \}$, we have $g(x_{\tilde{k}})=g(x_{\rho})$.} This implies that
\begin{align} \label{eq:positive-1}
\widetilde{\PP}_{x,\hat{\sigma},\tau}(\rho\leq \tilde{k})\widetilde{\EE}_{x,\hat{\sigma},\tau}\left[g(x_{\rho}) \mid \rho\leq \tilde{k}\right]=\widetilde{\EE}_{x,\hat{\sigma},\tau}\left[g\left(x_{\tilde{k}}\right)\right]\geq v_{n(x)}(x)\geq M.
\end{align}
}
Using the fact that the payoff function $g$ has maximal norm $1$, we deduce from (\ref{eq:positive-1}):
\begin{align}\label{eq:positive-2}
\widetilde{\PP}_{x,\hat{\sigma},\tau}(\rho\leq \tilde{k})\geq M.
\end{align}  
Define the strategy \footnote{{\color{violet} The strategy $\sigma^*$ is a generalized mixed strategy, which is equivalent to a behavior strategy by Kuhn's theorem.}} $\sigma^*$  as concatenations of $(\hat{\sigma}(x_{u_l}))_{l\geq 0}$ at the random stages $(u_{\ell})_{\ell\geq 0}$, where $u_{\ell}$ is defined inductively along the play {\color{violet2} by }$u_0=1$ and $u_{\ell+1}-u_\ell=\tilde{k}(x_{u_\ell})$ follows the uniform distribution over $\{1,...,n(x_\ell)\}$. {\color{violet}Let $\widetilde{\PP}_{x,\sigma^*, \tau}$ be the (product) probability distribution over $H_\infty\times \{1,...,n_0\}^\NN$ induced by $(x, \sigma^*,\tau)$, and $\widetilde{\EE}_{x,\sigma^*, \tau}$ the corresponding expectation. Let $\varepsilon>0$.} \\   
\noindent  $(i)$ \ We show that $\sigma^*$ is uniformly terminating. By (\ref{eq:positive-2}), the conditional probability of absorbing on each block $\{u_{l-1},...,u_{l}-1\}$ is no smaller than $M$. Thus for any $\tau$ and $x_1\in X^0$, $$\widetilde{\PP}_{x_1, \sigma^*,\tau}\left(\rho\geq u_{l}\right)\leq (1-M)^l, \ \forall l\geq 1.$$
The length of each block is uniformly bounded by $n_0$, thus if we put $l^* \geq \frac{\ln (\varepsilon)}{\ln(1-M)}$:
\begin{eqnarray} \label{eq4}
\widetilde{\PP}_{x_1,\sigma^*,\tau}\left( \rho\leq n_0l^*\right)\geq \widetilde{\PP}_{x_1,\sigma^*,\tau}\left( \rho\leq u_{l^*}\right)\geq 1-(1-M)^{l^*}\geq 1-\varepsilon.
\end{eqnarray}
\noindent $(ii)$ \ We now argue that $\sigma^*$ uniformly guarantees $v_{n(x_1)}(x_1)-3\varepsilon$. Let $N_0=n_0 l^*/\varepsilon$. 
For any $\tau\in\T$, $x_1\in X^0$ and $n\geq n_0  l^*$, we have
\begin{eqnarray*} 
\EE_{x_1,\sigma^*,\tau} \left[ g\left(x_n\right)\right]&=&{\color{orange2} \widetilde{\EE}_{x_1,\sigma^*,\tau} \left[ \sum_{l=0}^{\ell^*-1}  g(x_n)\1_{u_l \leq \rho <u_{l+1} }+ g(x_n) \1_{u_{\ell^*} \leq \rho} \right]}\\ 
 & = &\sum_{l=0}^{\ell^*-1} \widetilde{\PP}_{x_1,\sigma^*,\tau}(u_{l}\leq  \rho < u_{l+1})\widetilde{\EE}_{x_1,\sigma^*,\tau}\left[g(x_{\rho})|u_{l}\leq \rho < u_{l+1}\right] \\
&+& \widetilde{\PP}_{x_1,\sigma^*,\tau}(u_{l^*} \leq \rho)\widetilde{\EE}_{x_1,\sigma^*,\tau}\left[g(x_{n})|\rho\geq u_{l^*} \right].
\end{eqnarray*}
According to (\ref{eq4}), $\PP_{x_1,\sigma^*,\tau}(\rho\geq u_{l^*})\leq \varepsilon$, thus we focus on an absorption before $u_{\ell^*}$: 
\begin{eqnarray} \label{eq5}
\EE_{x_1,\sigma^*,\tau} \left[ g(x_n)\right]\geq\sum_{l=0}^{l^*-1} \widetilde{\PP}_{x_1,\sigma^*,\tau}(u_{l}\leq  \rho < u_{l+1})\widetilde{\EE}_{x_1,\sigma^*,\tau}\left[g(x_{\rho})|u_{l}\leq \rho < u_{l+1}\right]-\varepsilon
\end{eqnarray}
For each $l\geq 0$, $\sigma^*$ is following $\hat{\sigma}(x_{u_l})$ for $u_{l+1}-u_{l}=\tilde{k}(x_{u_\ell})$ {\color{orange2} stages}. Thus (\ref{eq:positive-1}) applies, and we obtain: for $\ell\geq 1$, 
\begin{eqnarray*}
 \widetilde{\PP}_{x_1,\sigma^*,\tau}(u_{l}\leq  \rho < u_{l+1}) \widetilde{\EE}_{x_1,\sigma^*,\tau}[g(x_{\rho})|u_{l}\leq \rho < u_{l+1}]\geq \widetilde{\PP}_{x_1,\sigma^*,\tau}(\rho>u_\ell) M>0,
\end{eqnarray*}
and for $l=0$,
$$\widetilde{\PP}_{x_1,\sigma^*,\tau}(1\leq  \rho < u_{1}) \widetilde{\EE}_{x_1,\sigma^*,\tau}[g(x_{\rho})|1\leq \rho < u_{1}]\geq v_{n(x_1)}(x_1).$$
\noindent By substituting the two previous inequalities into (\ref{eq5}), we obtain that 
\begin{eqnarray}\label{eq6}
\forall n\geq n_0l^*, \forall x_1\in X^0, \  \EE_{x_1,\sigma^*,\tau} \left[g(x_n)\right] \geq   v_{n(x_1)}(x_1)-\varepsilon.  
\end{eqnarray}
Now for $n\geq N_0$, we deduce that $\gamma_n(x_1,\sigma^*,\tau)   \geq v_{n(x_1)}(x_1)-3\varepsilon$.
\end{proof}\\

One can deduce from Proposition \ref{posstage} a first result on recursive games with the condition that the sequence of $n$-stage values converges uniformly to a function bounded away from $0$.

\begin{corollary}\label{poslim}
Assume that in a recursive game $\Gamma$, the sequence of $n$-stage values $(v_n)_{n\geq 1}$ converges uniformly to a function $v$ satisfying for every $x\in X^0$, $v(x) \geq M'>0$ for some $M'$. Then $\Gamma$ is positive-valued and player 1 uniformly guarantees $v$ with uniformly terminating strategies.
\end{corollary}


\subsection{Existence of the uniform value (proof of Theorem \ref{theo1})}\label{proof_uni}

This subsection is devoted to the proof of Theorem \ref{theo1}: the total boundedness of $\{v_n,n\geq 1\}$ implies the existence of the uniform value $v_\infty$. We prove that player 1 guarantees the point-wise limit superior value $x\mapsto v(x):=\limsup_{n} v_n(x)$. By symmetry, player $2$ guarantees $\liminf_n v_n(x)$, and the result follows.  \\

{\color{orange2}
The {\color{vert}uniform $\varepsilon$-optimal strategy} will use alternatively two different types of strategies. This approach is classical for recursive games and has been used for example in Rosenberg and Vieille \cite{Rosenberg_2000} and in Solan and Vieille \cite{Solan_2002}. Our construction is {\color{violet2} close} to Solan and Vieille \cite{Solan_2002} {\color{vert} in which some similar "positive-valued recursive game" is introduced to make a reduction for the general case}. 

{\color{rouge} The proof is decomposed into three parts. In the first one, we introduce a family of auxiliary positive-valued recursive games and define the first type of strategies. In the second part, we define the second type of strategies. Finally, we construct the strategy $\sigma^*$ and prove that it is uniform $\varepsilon$-optimal.}\\

{\color{vert} Before proceeding to the proof, let us first prove a preliminary result, which shows that due to the total boundedness of $\{v_n\}$, the point-wise limit superior of $(v_n)$ can be realized along uniform convergent subsequences.} {\color{jaune} We fix a recursive game $\Gamma$ for the rest of this section.}

\begin{proposition} \label{prop:limsup}
For every $x\in X$, we have
\[
v(x)=\limsup_n v_n(x)=\max_{f\in F} f(x),
\]
where $F$ is the set of limit points of the sequence $(v_n)_{n\geq 1}$ in $({\bf B}(X),\|.\|_\infty)$.
\end{proposition}

\begin{proof}
{\color{vert}  $(\textbf{B}(X),\|.\|_\infty)$ is a complete metric space and  ($\{v_n\},\|\cdot\|_\infty$) is totally bounded, therefore $F$ is compact and non-empty.  For every $x\in X$, we denote $w(x):=\max_{f \in F} f(x)$. Fix $x\in X$.  Since $v(x)$ is the largest limit point of $(v_n(x))_{n\geq 1}$, we have $w(x)\leq v(x)$.  By definition of the limit superior, there exists a subsequence $(v_{n_k}(x))_{k \geq 1}$ which converges to $\limsup v_n(x)$. There exists a subsequence of $(v_{n_k})_{k\geq 1}$ that converges in $({\bf B}(X),\|.\|_\infty)$ to some $f^* \in F$, therefore
\[
\max_{f\in F} f(x) \geq f^*(x)=v(x).
\]
} 
\end{proof}

\subsubsection{Reduction: auxiliary recursive games}\label{reduction}

\vspace{5mm}

\noindent \textbf{Auxiliary recursive games } \ {\color{rouge} Let $\theta: X \rightarrow \{0,1\}$. We define the auxiliary recursive game {\color{jaune} $\Gamma^\theta=\langle A, B, X=X^0_\theta\bigcup X^*_\theta, q_\theta,g_\theta\rangle$} where any active state $x\in X^0$ such that $\theta(x)=1$ is seen as an absorbing state: the active state space of $\Gamma^\theta$ is $X^0_\theta= \{ x\in X^0,\theta(x)=0\}$ and the absorbing state space is {\color{jaune}$X^*_\theta= X^* \bigcup \{ x\in X^0, \theta(x)=1 \}$}. The transition $q_\theta$ is equal to $q$ and the payoff $g_\theta$ is equal to $g$ {\color{jaune} on all states except $\{x\in X^0, \theta(x)=1\}$, on which the state is absorbing and the absorbing payoff is $g_\theta=v$}. For every $n\geq 1$, let $v^\theta_n$ be the value of the $n$-stage auxiliary game $\Gamma^\theta_n$.}

\begin{proposition}\label{posaux}
Let $\eta >0$ and $\theta:X\rightarrow \{0,1\}$. There exists $n_0 \geq 1$ such that for every $x_1 \in X_\theta^0$, there exists $n(x_1)\leq n_0$ with $v^\theta_{n(x_1)}(x_1) \geq v(x_1)-4\eta$.
\end{proposition}

\begin{proof} \ Let $\eta>0$ be fixed and $F_R=\{f_1,...,f_R\}\subseteq F$ be a finite cover of size $\frac{\eta}{2}$ of the set $F$. As $\{v_n, n\geq 1\}$ is totally bounded, there exists some stage $n(\eta)\in\NN$, after which any $n$-stage value $v_n$ is $\frac{\eta}{2}$-close to $F$ its set of accumulation points, hence $\eta$-close to $F_R$:
\begin{equation}\label{eqA}
\exists n(\eta)\in\NN, \ \forall n\geq n(\eta), \ \exists f_r\in\{f_1,...,f_R\}, \ s.t. \ \ ||v_n-f_r||_\infty\leq \eta;
\end{equation}
Moreover for every $r\in\{1,..,R\}$, $f_r$ is an accumulation point of $\{v_n, n\geq 1\}$, therefore there exists some  $n_r> \frac{n(\eta)}{\eta}$ such that $v_{n_r}$ is $\eta$-close to $f_r$:
\begin{equation}\label{eqB}
\forall f_r\in F_R, \ \exists n_r > \frac{n(\eta)}{\eta}, \ s.t. \ \ ||v_{n_r}-f_r||_\infty\leq \eta.
\end{equation}
Finally we take $n_0=\max\{n_r: 1\leq r\leq R\}$. {\color{rouge}
The integers ${n_r}$ are chosen such that when absorption in $X^*_{\theta}$ occurs in the game of length $n_r$ the remaining number of stages is either a fraction smaller than $\eta$ of the total length of the game or greater than $n(\eta)$ and Equations (\ref{eqA}) applies.}

Let $x_1\in X^0_\theta$ be any non-absorbing state in the auxiliary game $\Gamma^\theta$. By compactness of $F$, there exists $f\in F$ such that $f(x_1)=v(x_1)$ and $f_r\in F_R$ with $||f-f_r||_\infty\leq \frac{\eta}{2}$. In particular at state $x_1$,
\begin{equation*}
f_r(x_1)\geq f(x_1)-\frac{\eta}{2}= v(x_1)-\frac{\eta}{2}, 
\end{equation*}
which together with (\ref{eqB}) implies that
\begin{equation}\label{eq2eta}
v_{n_r}(x_1) \geq f_r(x_1)-\eta\geq v(x_1)-\frac{3}{2}\eta.  
\end{equation}


We now prove that
\begin{equation}\label{eq4eta}
v^\theta_{n_r}(x_1)\geq v_{n_r}(x_1)-2\eta.
\end{equation}
\noindent Denote by $$\rho_\theta=\inf_{t\geq 1}\{x_t\in X^*_\theta\}=\inf_{t\geq 1} \{x_t\in X^*  \text{ or } \theta(x_t)=1\}$$ the stopping time associated to absorption in $\Gamma^\theta$, and set $\rho_\theta^{n_r}=\min (\rho_\theta, n_r )$. {\color{rouge} An adaptation of standard proof technique of the Shapley equation} gives us:
\[
v_{n_r}(x_1)= \max_{\sigma \in \Sigma} \min_{\tau \in \T} \EE_{x_1,\sigma,\tau} \left( \frac{1}{n_r} \left( \sum_{t=1}^{\rho^{n_r}_\theta-1} g(x_t) \right) +\frac{n_r-\rho^{n_r}_\theta+1}{n_r} v_{n_r-\rho^{n_r}_\theta+1}(x_{\rho^{n_r}_\theta}) \right).
\]
\noindent We separate the histories into two sets depending on whether $n_r-\rho^{n_r}_\theta(h)+1 > n(\eta)$ in which cases Equation (\ref{eqA}) applies, or $n_r-\rho^{n_r}_\theta(h)+1 \leq n(\eta)$ in which cases $\frac{n_r-\rho^{n_r}_\theta(h)+1}{n_r} \leq \eta$ \big(by definition $n_r\geq \frac{n(\eta)}{\eta}$\big), and deduce that
\begin{align*}
v_{n_r}(x_1) \leq \max_{\sigma \in \Sigma} \min_{\tau \in \T } \EE_{x_1,\sigma,\tau}\left( \frac{1}{n_r} \sum_{t=1}^{\rho^{n_r}_\theta-1} g(x_t) +\frac{n_r-\rho^{n_r}_\theta+1}{n_r} f'_h(x_{\rho^{n_r}_\theta}) \right) + 2\eta,
\end{align*}
\noindent with $f'_h\in F_R$ depending on the history given by Equation ($\ref{eqA}$) applied to $v_{n_r-\rho^{n_r}_\theta+1}$ when $n_r-\rho^{n_r}_\theta+1> n(\eta)$, and any function in $F_r$ otherwise. Therefore, by considering $v$ as the supremum of $f\in F$ at each point $x_{\rho_\theta^{n_r}}\in X$, we have $f'_h(x_{\rho_\theta^{n_r}})\leq v(x_{\rho_\theta^{n_r}})$, thus
\begin{align*}
v_{n_r}(x_1) & \leq \max_{\sigma \in \Sigma} \min_{\tau \in \T } \EE_{x_1,\sigma,\tau}\left( \frac{1}{n_r} \sum_{t=1}^{\rho^{n_r}_\theta-1} g(x_t) +\frac{n_r-\rho^{n_r}_\theta+1}{n_r} v(x_{\rho^{n_r}_\theta}) \right) + 2\eta\\
& = v_{n_r}^\theta(x_1)+2\eta.
\end{align*}
This proves inequality (\ref{eq4eta}). We now use Equation (\ref{eq2eta}) and Equation (\ref{eq4eta}) to conclude:
\[
{\color{rouge} v^\theta_{n_r}(x_1)\geq v(x_1)-4\eta.}
\]
It means that for each $x_1\in X_\theta^0$, there exists $n(x_1):=n_r\leq n_0=\max\{n_r:\ 1\leq r\leq R\}$, such that {\color{rouge} $v^\theta_{n(x_1)}(x_1)\geq v(x_1)-4\eta$}.\end{proof}.  

{\color{violet2}

\begin{remarque}
Proposition \ref{posaux} is also true if $\theta$ is a deterministic stopping time and not only a function on the state. The auxiliary game would be defined on a larger state space: the set of finite histories of the original game. The proof in itself is similar.
\end{remarque}

Fix now any $\varepsilon>0$ and define $\theta_\varepsilon: X\rightarrow \{0,1\}$ such that {\color{jaune} $\{x\in X,\ \theta_\varepsilon(x)=1\}=\{x \in X, v(x)< \varepsilon\}$}. We denote by {\color{jaune} $\Gamma^\varepsilon=\langle A, B, X=X^0_\varepsilon\bigcup X^*_\varepsilon, q_\varepsilon,g_\varepsilon\rangle$} {\color{jaune} the auxiliary game associated to $\Gamma$ defined by the stopping time $\theta_\varepsilon$}. 
}
\begin{corollary}\label{mart}
In the game $\Gamma^\varepsilon$, Player 1 uniformly guarantees $v$  with uniformly terminating strategies that depend only on past states.
\end{corollary}

\begin{proof}
{\color{rouge}
Let $\eta \in (0,\varepsilon/8]$, by Proposition \ref{posaux} there exists $n_0 \geq 1$ such that for every $x_1 \in X_\varepsilon^0$, there exists $n(x_1)\leq n_0$ with
\begin{align}\label{toto2}
v^\varepsilon_{n(x_1)}(x_1) \geq v(x_1)-4\eta \geq \varepsilon/2,
\end{align}
where the second inequality comes from the definition of $X_\varepsilon^0$. Therefore, $\Gamma^\varepsilon$ is a positive-valued recursive game and by Proposition \ref{posstage}, player $1$ uniformly guarantees $v^\varepsilon_{n(\cdot)}(\cdot)$ with uniformly terminating strategies in $\widehat{\Sigma}$. By Equation (\ref{toto2}), it follows that for every $\eta>0$, player $1$ uniformly guarantees $v -4\eta$ with uniformly terminating strategies.
}
\end{proof} \\

{\color{vert} Fix now a strategy $\sigma^*_\varepsilon$ that is uniformly terminating in $\Gamma^{\varepsilon}$, depends only on past states and guarantees $v(x_1)-\varepsilon^2$ in $\Gamma^\varepsilon(x_1)$ for every $x_1 \in X^0_{\varepsilon}$.}

\subsubsection{One-shot game}\label{osg}

\noindent \textbf{One-shot game } $G^f$ \  \ For each $f:X\to[-1,+1]$ and $x_1 \in X$, we define the one-shot game {\color{vert}$G^f$} as follows: player 1's {\color{vert} action set is $A(x_1)$, player 2's action set is $B(x_1)$}, and the payoff is for each $(s,t)\in \De\big(A \big)\times \De \left(B \right)$,
\begin{equation*}
{\color{vert}\EE_{q(x_1,s,t)}[f(x_2)]}=\sum_{a\in A, b\in B}s(a)t(b)\left(\sum_{x_2\in X}q(x_1,a,b)(x_2)f(x_2)\right).
\end{equation*}

\begin{lemme}\label{oneshotV}
For any limit point $f\in F$, the one-shot game $G^f$ has a value equal to $f$.
\end{lemme}

{\color{bleu}

\begin{proof}
Let $n\geq 1$, {\color{vert}it is known that} (cf. Vigeral \cite{Vigeral_2009} p.40, Lemma 4.2.2)
\[
\|v_n-v_{n+1}\|_\infty \leq \frac{2}{n+1},
\]
and by Shapley's formula that
\begin{align*}
v_{n+1}(x_1) & = \sup_{s \in \De\big(A(x_1)\big)} \inf_{t \in \De \left( B(x_1) \right) } \EE_{q(x_1,s,t)} \left[ \frac{1}{n+1} g(x_1) +\frac{n}{n+1} v_{n}(x_{2}) \right]\\
                                 & = \inf_{t \in \De \left( B(x_1) \right)}  \sup_{s \in \De\big(A(x_1)\big)} \EE_{q(x_1,s,t)} \left[ \frac{1}{n+1} g(x_1) +\frac{n}{n+1} v_{n}(x_{2}) \right].
\end{align*}

{\color{rouge} We obtain the result }by taking the limit along a subsequence converging uniformly to $f\in F$. 
\end{proof}\\

{\color{vert} Following Proposition \ref{prop:limsup}, one can take for each $x\in X$ some $f^* \in F$ such that $v(x)=f^*(x)\geq f(x), \forall f\in F$. Then the following result is a direct consequence of Lemma \ref{oneshotV} }. }

\begin{corollary}\label{coro:one-shot}
For every $x_1\in X$, there exists $s^*(x_1) \in \Delta(A(x_1))$ such that
\[
\forall b\in B(x_1), \ \EE_{q(x_1,s^{*}(x_1),b)} \left[ v(x_{2}) \right] \geq v(x_1).
\]
\end{corollary}

%
%
%
%
%

{\color{rouge} Fix now $s^*:=\big(s^*(x_1)\big)_{x_1\in X}$ a profile of strategies satisfying the conclusion of Corollary \ref{coro:one-shot}.}

\subsubsection{Optimal strategy}

{\color{vert} Roughly speaking, we build $\bar{\sigma}$ a uniform $\varepsilon$-optimal strategy for player $1$} to play $\sigma^*_\varepsilon$ in $\Gamma^\varepsilon$ on the states with value $v$ above $2\varepsilon$, and to play $s^*$ in $G^v$ on the states with value $v$ below $\varepsilon$. And for the states with value $v$ between $\varepsilon$ and $2\varepsilon$, $\bar{\sigma}$ will be either of the two depending on the regime. \\

\noindent \textbf{Construction of $\bar{\sigma}$}\ \  Define a sequence of stopping times $(u_l)_{l\geq 1}$ and the concatenated strategy $\overline{\sigma}:=s^*u_1\sigma^*_\varepsilon u_2 s^* u_3 \sigma^*_\varepsilon u_4\cdot\cdot\cdot$ in $\Gamma$ as follows:

\begin{itemize}
\item 
$\overline{\sigma}$ is to play $s^{*}(x_n)$ at each stage $n$ up to stage (not included)
\[
u_1=\inf \{ n\geq 1, \ v(x_n) > 2\varepsilon\};
\]
and then to play $\sigma^*_\varepsilon(x_{u_1})$ up to stage (not included)
\[
u_2=\inf \{ n\geq u_1, \ v(x_n) < \varepsilon\}.
\]

\item In general: for each $r\geq 1$, $\overline{\sigma}$ is to play $\sigma_\varepsilon^*(x_{u_{2r-1}})$ from stage $u_{2r-1}$ ({\color{vert} \textit{the odd phase}}) up to stage (not included) 
\[
u_{2r}=\inf \{ n\geq u_{2r-1}, \ v(x_n) < \varepsilon\}.
\]
and then to play $s^{*}(x_{n})$ at each stage $n\geq u_{2r}$ ({\color{vert} \textit{the even phase}}), up to stage (not included)
\[
u_{2r+1}=\inf \{ n\geq u_{2r}, \ v(x_n) > 2\varepsilon \}.
\]
\end{itemize}
%
 


\begin{remarque}  {\color{jaune}
The idea of alternating between two types of strategies is common in Rosenberg and Vieille \cite{Rosenberg_2000}, Solan and Vieille \cite{Solan_2002} and this article. The main difference is  the definition of the target function $v$ used to define how to switch from one type of strategies to the other. Rosenberg and Vieille \cite{Rosenberg_2000} use the limit of discounted values and $\sigma^*_\varepsilon$ is an optimal strategy in some $\lambda$-discounted game (for $\lambda$ close to zero). Solan and Vieille \cite{Solan_2002} use the limsup value and introduce an auxiliary positive-valued game. We adopt a similar approach to Solan and Vieille \cite{Solan_2002} but with $v$ the largest limit point of $(v_n)$.
}

\end{remarque}

{\color{vert} \noindent\ By construction, $\bar{\sigma}$ depends on the histories only through the states and not the actions. Let us show that $\bar{\sigma}$ uniformly guarantees $v-25\varepsilon$ for player $1$, which finishes the proof of Theorem \ref{theo1}. }\\

{\color{vert} Fix from now on any $x_1\in X$}. Recall that $\rho$ denotes the absorption time in the game $\Gamma$. The next result shows that the process $(v(x_{\min(\rho,u_{l})}))_{l\geq 1}$, which is the value of $v$ at switching times $(u_l)$, is almost a submartingale up to an error of $\varepsilon^2$.%

%
\begin{proposition}\label{submartingale}
For every $l\geq 1$ and every $\tau \in \T$:
\begin{equation*}
\EE_{x_1,\overline{\sigma},\tau}[v(x_{\min(\rho,u_{l+1})})|\mathcal{H}_{\min(\rho,u_l)}]\geq v(x_{\min(\rho,u_l)})-\varepsilon^2\mathds{1}_{\rho>u_l},
\end{equation*}

\noindent on the event $\min(\rho,u_l)<+\infty$.\\
\end{proposition}

\begin{proof} Take any $\tau$ in $\T$. The result is true if $\rho\leq u_l$. Suppose that $l$ is even and $\rho>u_l$: by construction the strategy $(s^*(x_n))$ is used during the phrase $n\in\{u_l,...,u_{l+1}-1\}$, thus:
\begin{equation*}
\EE_{x_1,\overline{\sigma},\tau}[v(x_{n+1})|\mathcal{H}_{n}]\geq v(x_n)
, \textit{ for all } u_l\leq n<\min(\rho,u_{l+1}).
\end{equation*}
Therefore $(v(x_n))$ is a bounded submartingale and by Doob's stopping theorem,
\begin{equation*}
\EE_{x_1,\overline{\sigma},\tau}[v(x_{\min(\rho,u_{l+1})})|\mathcal{H}_{\min(\rho,u_l)}]\geq v(x_{\min(\rho,u_l)}).
\end{equation*}
\noindent Suppose that $l$ is odd and $\rho>u_l$. By construction, player $1$ is using  $\sigma_\varepsilon^*(x_{u_l})$, which uniformly guarantees $v(x_{u_l})-\varepsilon^2$ in the auxiliary game $\Gamma^\varepsilon(x_{u_l})$:
\begin{align}\label{eq:toto_1}
\exists N_0\geq 1, \ \ \EE_{x_1,\bar{\sigma},\tau}\left[\frac{1}{n}\sum_{t=u_l+1}^{u_l+n} g_\varepsilon(x_t)|\mathcal{H}_{u_l}\right]\geq v(x_{u_l})-\varepsilon^2 \text{ for all } n\geq N_0.
\end{align}
\noindent Denote by $\rho^\varepsilon=\min\{m\geq u_l+1: x_m\in X^*_\varepsilon\}$ the absorption time in $\Gamma^\varepsilon(x_{u_l})$. Since in recursive games the payoff is zero before absorption, we have
\begin{equation} \label{eq:toto_2}
\EE_{x_1,\bar{\sigma},\tau}\left[ g_\varepsilon(x_{\rho_\varepsilon})|\mathcal{H}_{u_l} \right] = \EE_{x_1,\bar{\sigma},\tau}\left[\lim_{n\to\infty} \frac{1}{n} \sum_{t=u_l+1}^{u_l+n} g_\varepsilon(x_t)|\mathcal{H}_{u_l}\right].
\end{equation}
\noindent By the dominated convergence theorem,
\begin{align}\label{eq:toto_3}
\EE_{x_1,\bar{\sigma},\tau}\left[\lim_{n\to\infty} \frac{1}{n} \sum_{t=u_l+1}^{u_l+n} g_\varepsilon(x_t) |\mathcal{H}_{u_l}\right] &= \lim_{n\to\infty} \EE_{x_1,\bar{\sigma},\tau}\left[\frac{1}{n} \sum_{t=u_l+1}^{u_l+n} g_\varepsilon(x_t) |\mathcal{H}_{u_l}\right].
\end{align}
\noindent We deduce from (\ref{eq:toto_1})-(\ref{eq:toto_3}) that
\begin{equation*}
 \EE_{x_1,\bar{\sigma},\tau}\left[ g_\varepsilon(x_{\rho^\varepsilon})|\mathcal{H}_{u_l}\right]\geq v(x_{u_l})-\varepsilon^2.
\end{equation*}
\noindent Moreover, $g_\varepsilon(x_{\rho^\varepsilon})=v(x_{\rho^\varepsilon})$ and conditionally on $\rho>u_l$, $\rho^\varepsilon=\min(u_{l+1},\rho)$. It follows that
%
%
\begin{align*}
\EE_{x_1,\overline{\sigma},\tau}[v(x_{\min(\rho,u_{l+1})})|\mathcal{H}_{u_l}]\geq v(x_{u_l})-\varepsilon^2.
\end{align*}

\end{proof}

\vspace{5mm}

\noindent Due to the possible error term $\varepsilon^2$, the sequence $(v(x_{\min(\rho,u_{l})}))_{l\geq 1}$ is not a submartingale. Nevertheless, one can prove a lemma similar to the usual upcrossing lemma for submartingale. \textcolor{bleu}{Indeed, the value is a martingale excepts if it crosses upwards the interval $[\varepsilon,2\varepsilon]$. When this happens, the value may decreases of at most $\varepsilon^2$.} 
With the submartingale property established in Proposition \ref{submartingale}, an easy adaptation of the standard result on upcrossing number of submartingale implies the following result, as was shown in Proposition 3 of Rosenberg and Vieille \cite{Rosenberg_2000}:

\begin{lemme}\label{upcross} Let $N=\sup\{p\geq 1: u_{2p-1}<+\infty\}$ be the number of times the process $(v(x_{u_l}))$ crosses upward the interval $[\varepsilon,2\varepsilon]$.For every $\tau \in \T$,
\begin{equation*}
\EE_{x_1,\overline{\sigma},\tau}[N]\leq\frac{1}{\varepsilon-\varepsilon^2}.
\end{equation*}
\end{lemme}

\noindent By construction, $\sigma_\varepsilon^*$ is \textit{uniformly terminating} within the auxiliary absorbing states $X^*_\varepsilon$. That is to say, any play between stages $u_{2p-1}$ and $u_{2p}$ ({\color{orange2} on an odd phase}) has bounded length with high probability under the strategy $\sigma_\varepsilon^*(x_{u_{2p-1}})$, uniformly over any starting state $x_{u_{2p-1}}\in X^0_\varepsilon$. Since Lemma \ref{upcross} implies that {\color{orange2} the number of odd phases} is bounded in expectations, the total frequency of stages on all {\color{orange2} odd phases} is negligible for $n$ large. Let us formalize this fact.

{\color{vert} Recall that $\rho^\varepsilon$ denotes the absorption time in the auxiliary game $\Gamma^\varepsilon$}. It follows that there exists $N_1>0$ such that
\begin{equation}\label{eq:N_1}
\forall x\in X^0_\varepsilon \text{ and } \tau\in\mathcal{T}: \PP_{x,\sigma^*_\varepsilon(x),\tau}(\rho^\varepsilon>N_1)\leq \varepsilon^3.
\end{equation}
\noindent For each $n\in\NN$, define $A_n=\{u_{2p-1}\leq n< \min(\rho,u_{2p}),u_{2p-1}<\rho, \text{ for some } p\}\subseteq H_\infty$. These are all infinite plays where stage $n$ is {\color{vert} in an odd phrase}, $i.e.$, the stages between $u_{2p-1}$ and $u_{2p}$ on which $\sigma^*_\varepsilon(x_{u_{2p-1}})$ is used. {\color{vert}We fix for the rest of subsection the uniform stage number $N_1$ satisfying (\ref{eq:N_1})}. 

\begin{lemme}\label{sum_An}
For every $\tau\in\mathcal{T}$ and every $n\geq \frac{N_1}{\varepsilon^3}$,
\begin{equation*}
\frac{1}{n}\sum_{k=1}^n\PP_{x_1,\overline{\sigma}(x_1),\tau}(A_k)\leq 5\varepsilon.
\end{equation*}
\end{lemme}

{\color{vert}
The proof for this lemma relies on the upcrossing property established in Lemma \ref{upcross}, and takes the same form as Lemma 27 in Solan and Vieille \cite{Solan_2002}. 
Solan and Veille \cite{Solan_2002} make some \textit{finiteness} assumption (on the set of non-absorbing states on which the target function is not bounded away from zero) in order to obtain the existence of {\color{violet2} $X^1_\varepsilon$} a subset of $X^0_\varepsilon$ and a uniform bound $N_1\geq 1$ such that 
\begin{equation*}
\forall x\in X^1_\varepsilon \text{ and } \tau\in\mathcal{T}: \PP_{x,\sigma^*_\varepsilon(x),\tau}(\rho^\varepsilon>N_1)\leq \varepsilon^3.
\end{equation*}
Under the assumption that $\{v_n, n\geq 1\}$ is totally bounded, we showed in Section \ref{reduction} \big(cf. the condition defined in (\ref{eq:N_1})\big) that we can consider \textcolor{violet2}{$X^1_\varepsilon$} to be the whole set $X^0_\varepsilon$. 
}\\

%

\noindent The following result is a reformulation of the submartingale property in Lemma \ref{submartingale}.
\begin{lemme}\label{control_stage}
For any $m_0\geq 1$, we have
\begin{equation*}
\EE_{x_1,\overline{\sigma},\tau}[v(x_{m_0})]\geq v(x_1)-\varepsilon^2\cdot \EE_{x_1,\overline{\sigma},\tau}[N]-2\PP_{x_1,\overline{\sigma},\tau}(A_{m_0})-\varepsilon.
\end{equation*}
\end{lemme}
\begin{proof} For a proof, we refer to Proposition 28 in Solan and Vieille \cite{Solan_2002}, where our lemma is stated as Equation (4) in their proof. \end{proof}\\

Now we use Lemma \ref{upcross}, Lemma \ref{sum_An} and Lemma \ref{control_stage} to prove the following proposition, which concludes the proof of Theorem \ref{theo1}.

\begin{proposition}\label{uniforme}
For any $x_1\in X^0$ and for any $\tau$,
$$\EE_{x_1,\bar{\sigma},\tau}\left[\frac{1}{n}\sum_{m=1}^ng(x_m)\right]\geq v(x_1)-25\varepsilon, \ \forall n\geq \frac{N_1}{\varepsilon^3}.$$
\end{proposition}

\begin{proof}  Take $x_1\in X^0$ and fix any $\tau$. In this proof $h$ will denote a pure play. We use the fact that $g(x_m)\geq v(x_m)-2 \varepsilon$ if $h\notin A_m$: indeed, either the play has absorbed so $g(x_m)=v(x_m)$, or we have $v(x_m) < 2\varepsilon$ and $g(x_m)=0$. Moreover, if $h\in A_m$, we use $g(x_m) \geq -1$. This gives us:
\begin{equation}\label{eq:gamma_n-0}
\begin{aligned}
\EE_{x_1,\overline{\sigma},\tau}\left[\frac{1}{n}\sum_{m=1}^n g(x_m)\right] 
&\geq \frac{1}{n}\EE_{x_1,\overline{\sigma},\tau}\left[\sum_{m=1}^n\mathds{1}_{h\notin A_m} (v(x_m)-2\varepsilon) + \sum_{m=1}^n \mathds{1}_{h\in A_m} (-1) \right]\\
& \geq \frac{1}{n}\EE_{x_1,\overline{\sigma},\tau}\left[\sum_{m=1}^n v(x_m) \right]  + \frac{1}{n}\EE_{x_1,\overline{\sigma},\tau}\left[\sum_{m=1}^n \mathds{1}_{h\in A_m}\big(-1-(v(x_m)-2\varepsilon)\big) \right]- 2\varepsilon.
\end{aligned}
\end{equation}
\noindent  Lemma \ref{control_stage}  (taking average sum on $m_0=1,...,n$) implies that
\begin{align}\label{eq:gamma_n-1}
\frac{1}{n}\EE_{x_1,\overline{\sigma},\tau}\left[\sum_{m=1}^n v(x_m) \right] & \geq v(x_1) - \varepsilon^2 \cdot \EE_{x_1,\overline{\sigma},\tau}[N] - \frac{2}{n}\sum_{m=1}^n \PP_{x_1,\overline{\sigma},\tau}(A_{m})-\varepsilon.
\end{align}
Moreover, the bound $v(x_m)\leq 1$ gives
\begin{equation}\label{eq:gamma_n-2}
\begin{aligned}
\frac{1}{n}\EE_{x_1,\overline{\sigma},\tau}\left[\sum_{m=1}^n \mathds{1}_{h\in A_m}\big(-1-v(x_m)+2\varepsilon\big) \right]  & \geq \frac{1}{n}\EE_{x_1,\overline{\sigma},\tau}\left[\sum_{m=1}^n \mathds{1}_{h\in A_m}(-2+2\varepsilon) \right] \\
& = (-2+2\varepsilon)\frac{1}{n}\sum_{m=1}^n \PP_{x_1,\overline{\sigma},\tau}(A_{m}).
\end{aligned}
\end{equation}
We substitute (\ref{eq:gamma_n-1}) and (\ref{eq:gamma_n-2}) back into (\ref{eq:gamma_n-0}) to obtain
\begin{align*}
\EE_{x_1,\overline{\sigma},\tau}\left[\frac{1}{n}\sum_{m=1}^n g(x_m)\right] &\geq v(x_1)-\varepsilon^2 \cdot \EE_{x_1,\overline{\sigma},\tau}[N]-3\varepsilon+(-4+2\varepsilon) \cdot \left(\frac{1}{n} \sum_{m=1}^n \PP_{x_1,\overline{\sigma},\tau}(A_{m})\right).
\end{align*}
Finally, we use Lemma \ref{upcross} and Lemma \ref{sum_An} in the equality to have that: $\forall n\geq \frac{N_1}{\varepsilon^3}$ and $\forall \varepsilon\leq \frac{1}{2}$,
\begin{align*}
\EE_{x_1,\overline{\sigma},\tau}\left[\frac{1}{n}\sum_{m=1}^n g(x_m)\right]
&\geq v(x_1)-\frac{\varepsilon^2}{\varepsilon-\varepsilon^2}-3\varepsilon-20\varepsilon\geq v(x_1)-25\varepsilon.
\end{align*}
\noindent note that $N_1$ does not depend on the particular choice of $x_1$ in $X^0$, so the strategy $\overline{\sigma}$ uniformly guarantees $v-25\varepsilon$ in the infinite game $\Gamma$. \end{proof}

\subsection{Pure optimal strategy (proof of Corollary \ref{theo3})}\label{pure_proof}

{\color{violet} To prove the result, it is sufficient to show that both the strategy $s^*$ and the strategy $\sigma^*_\varepsilon$ defined {\color{jaune} in the proof of Theorem \ref{theo1} can be chosen pure and depending only on the history of states}.

By assumption, the $n$-stage game $\Gamma_n(x)$ has a value in pure strategies. It follows that Shapley's equation for any $v_n$ is satisfied with pure strategies, and so is Lemma \ref{oneshotV}. We deduce that \textcolor{orange2}{there exists a pure action $s^*$ that satisfies the conclusion of Corollary \ref{coro:one-shot}}.
 
The construction of the strategy $\sigma^*_\varepsilon$ appeared in the proof of Proposition \ref{proof_uni}, where it was defined as the concatenation of a sequence of strategies $\big(\hat{\sigma}(x_{u_\ell})\big)_{\ell\geq 1}$ at the random stages $(u_\ell)_{\ell\geq 1}$. As each $\hat{\sigma}(x)$ is optimal in the $n(x)$-stage game $\Gamma_{n(x)}(x)$, $\hat{\sigma}(x_{u_\ell})$ can be taken pure. The definition of the random stages $u_\ell$ involved a randomized stopping time $\tilde{k}\in\{1,...,n(x)\}$ satisfying: 
\[
\forall \tau \in \T,\  \widetilde{\EE}_{x,\hat{\sigma},\tau} \left[ g(x_ {\tilde{k}}) \right] \geq \min_{\tau' \in \T}  \EE_{x,\hat{\sigma},\tau'} \Big[\frac{1}{n(x)} \sum_{t=1}^{n(x)} g(x_t) \Big].
\]
To obtain a pure strategy $\sigma^*_\varepsilon$, we show that the random stopping time $\tilde{k}$ can be replaced by a stopping time (pure one), which depends only on the history of states and not on the actions. {\color{rouge} In order to build this stopping time, we restrict ourselves to strategies in $\widehat{\Sigma}$ , $i.e.$, strategies which depend only on past states. Note that each $\hat{\sigma}(x_{u_\ell})$, as an optimal strategy in $\Gamma_{n(x_{u_\ell})}(x_{u_\ell})$, can be taken in $\widehat{\Sigma}$.}
%


\begin{lemme}\label{lem:reduced} \ Fix any $\hat{\sigma}\in\widehat{\Sigma}$ and $x_1$. For any $\tau\in\T$, there exists some $\hat{\tau}\in \widehat{\T}$ such that $\PP_{x_1,\hat{\sigma},\hat{\tau}}(x_1,...,x_t)=\PP_{x_1,\hat{\sigma},\tau}(x_1,...,x_t)$ for any $(x_1,...,x_t)\in X^t, t\geq 1$. 
\end{lemme}
\begin{proof} For all $t\geq 1$, we denote by $s_t:=(x_1,...,x_t)$ the $t$ first states. For any $\tau\in\T$, define the reduced strategy $\hat{\tau}\in\widehat{\T}$ as:
$$\hat{\tau}_t(s_t)=\sum_{h_t\in H_t(s_t)}\PP_{x_1,\hat{\sigma},\tau}(h_t|s_t)\tau_t(h_t), \ \forall s_t, \ \forall t\geq 1. $$ 
where $H_t(s_t)$ denotes the histories in $H_t$ containing $s_t$. Then we obtain by definition:
$$\PP_{x_1,\hat{\sigma},\hat{\tau}}(s_{t+1})=\sum_{h_t\in H_t(s_t)}\PP_{x_1,\hat{\sigma},\tau}(h_t|s_t)\PP_{x_1,\hat{\sigma},\tau}(s_{t+1}|h_t)=\PP_{x_1,\hat{\sigma},\tau}(s_{t+1}).$$ 
\end{proof}
}

{\color{violet} \begin{lemme}\label{optstop}  Fix any $x_1\in X^0$ and $\sigma\in\widehat{\Sigma}$. For any $n\geq 1$, there exists a stopping time $\theta:\bigcup_{1\leq t\leq n}X^t\to \{1,...,n\}$ such that for every strategy $\tau$ of player 2:

\begin{equation*}
\EE_{x_1,\sigma,\tau}\left[g(x_\theta) \right]\geq \min_{\tau'}\EE_{x_1,\sigma,\tau'}\left[\frac{1}{n}\sum_{t=1}^n g(x_t)\right].
\end{equation*}
\end{lemme} 
\begin{proof} By Lemma \ref{lem:reduced}, we can assume that $\tau\in\widehat{\T}$. Let us prove the result by induction. {\color{jaune} For every $x_1\in X^0$, the result is true for $n=1$. Suppose that the claim is true for $n-1$. Let $x_1 \in X^0$. By applying the inductive assumption to the different states possible at stage $2$, we obtain that there is some stopping time  $\theta^+:\bigcup_{t=1}^{n-1} X^t\to \{2,...,n\}$ such that}
\begin{align}\label{eq:theta+}
\EE_{x_1, \sigma,\tau}\big[g(x_{\theta^+})|x_2\big]\geq \min_{\tau'}\EE_{x_1,\sigma,\tau'}\Big[\frac{1}{n-1}\sum_{t=2}^{n}g(x_t)\big|x_2\Big]:=w_{n-1}(\sigma,x_1,x_2).
\end{align}
Denote $w_{n-1}(\sigma,x_1)=\inf\limits_{y\in\De(J)} \EE_{x_1,\sigma,y}\big[w_{n-1}(\sigma,x_1,x_2)\big]$. We define the stopping time $\theta:\bigcup_{t=1}^{n} X^t\to \{1,...,n\}$ by
\[
\forall (x_1,..,x_t)\in X^t,\ \ \theta(x_1,...,x_t)=\begin{cases} 1 & \text{ if }0\geq w_{n-1}(\sigma,x_1), \\
\theta^+(x_2,...,x_t) & \text{ otherwise} .
\end{cases}
\]

%

\noindent According to the definition of $\theta$ and the inductive assumption (\ref{eq:theta+}) for $\theta^+$:
\begin{align*}
\EE_{x_1,\sigma,\tau}[g(x_\theta)]&= g(x_1)\mathds{1}_{0\geq w_{n-1}(\sigma,x_1)}+ \EE_{x_1,\sigma,\tau}\Big[\EE_{x_1,\sigma,\tau}\big[g(x_{\theta^+})|x_2\big]\Big]\mathds{1}_{0< w_{n-1}(\sigma,x_1)}\\
& \geq \max\big\{0, w_{n-1}(\sigma,x_1)\big\}\\ 
& \geq \frac{n-1}{n}w_{n-1}(\sigma,x_1). 
\end{align*}
Finally $g(x_1)=0$, therefore
$$\frac{n-1}{n}w_{n-1}(\sigma,x_1)= \left(\frac{n-1}{n}\right)\inf_{\tau'}\EE_{x_1,\sigma,\tau'}\left[\frac{1}{n-1}\sum_{t=2}^{n-1}g(x_t)\right]=\min_{\tau'}\EE_{x_1,\sigma,\tau'}\left[\frac{1}{n}\sum_{t=1}^n g(x_t)\right]. $$
This concludes the inductive proof. \end{proof} 

\begin{remarque} Let $\Gamma$ be a {\color{rouge} stochastic game} where the payoff function depends only on the state but not the actions, the proof for the above result follows the same way.   
\end{remarque}
}

\section{Application to {\color{orange2} {\color{rouge} recursive games with signals}}} \label{sec:5}

In this last section, we apply our result to the model of finite {\color{orange2} {\color{rouge} recursive games with signals}} where one player is more informed than the other player. Introducing an auxiliary stochastic game similar to the one defined in Gensbittel \textit{et al}. \cite{Gensbittel_2014}, we show that the study of such a recursive game can be reduced to the study of recursive game with a countable state space satisfying the assumption of Corollary \ref{theo3}.



\subsection{Model}

{\color{jaune} The following model of general repeated game is introduced in  Mertens {\it et al.} \cite{Mertens_2015}}. A repeated game $\Gamma=(K,I,J,C,D,g,q)$ is given by
\begin{itemize}
\item a finite state space: $K$.
\item two finite action spaces $I$ and $J$.
\item two finite signal spaces $C$ and $D$.
\item a payoff function: $g:K \times I \times J\to[-1,+1]$.
\item a transition probability function (on states and signals): $q$ from $ K \times I \times J$ to $\Delta(K\times C\times D)$.
\end{itemize}

\noindent Denote by $\Gamma(\pi)$ the game {\color{vert}with} an initial probability distribution $\pi \in\De( K\times C \times D)$, {\color{vert}which is played as follows. Initially, the triple $(k_1,c_1,d_1)$ is drawn according to $\pi$. At stage 1: player 1 learns $c_1$ and player 2 learns $d_1$. Then simultaneously player 1 chooses an action $i_1\in I$ and player 2 chooses an action $j_1\in J$. The stage payoff is $g(k_1,i_1,j_1)$, and the new triple $(k_2,c_2,d_2)$ is drawn according to $q(k_1,i_1,j_1)$. The game then proceeds to stage 2: player 1 observes $c_2$, and player 2 observes $d_2$ etc...}    \\

We assume that each player's signal contains his own action. Formally, there exists $\hat{\imath}:C \rightarrow I$ and $\hat{\jmath}:D \rightarrow J$ such that
\[
\forall k \in K,\ \sum_{k',c,d} q\big(k,\hat{\imath}(c),\hat{\jmath}(d)\big)(k',c,d)=1.
\]

{\color{vert}We will focus on repeated games with the following two features: \textit{recursive} and \textit{one player is more informed than the other}.}

\begin{definition}
The repeated game $\Gamma$ is \emph{recursive} if there exist $K^0$ and $K^*$, a partition of $K$ such that:
\begin{itemize}
\item the stage payoff is $0$ on active states: $\forall(k,i,j) \in K^0\times I \times J$, $g(k,i,j)=0$.
\item states in $K^*$ are absorbing: $\forall k\in K^*$, $\sum_{c\in C,d\in D}q(k,i,j)(k,c,d)=1$ for all $(i,j)\in I \times J$ and $g(k,i,j)$ depends only on $k$.
\end{itemize}
{\color{rouge} In the rest of the paper, a recursive repeated game will be called a \textit{recursive games with signals.}}
\end{definition}

\begin{definition}
Player 1 is \emph{more informed} than player 2 in the recursive game $\Gamma$ if there exists a mapping $\hat{d}:C \rightarrow D$ such that, if $E$ denotes $\{(k,c,d)\in K \times C \times D,\ \hat{d}(c)=d \}$, then
\[q(k,i,j)(E)=1, \ \forall (k,i,j)\in K\times I \times J. \]
\end{definition}
 {\color{vert} \begin{notation} We denote by: $\Delta^1(K\times C \times D)=\{\pi| \pi(E)=1\}$. \end{notation} }
 
We define similarly that player $2$ is more informed than player $1$.  Whenever player $1$ is more informed than player $2$ and player $2$ is more informed than player $1$, {\color{vert} $\Gamma$ is a repeated game with \emph{symmetric signals}. We denote by $\Delta^*(K\times C\times D)$ the set of symmetric initial distributions.} \\

\begin{remarque}  \ By assumption, if player 1 is more informed than player $2$, he learns especially the action played by player $2$ since it is included in the signal of player $2$. Player 2 is in general not informed of the action played by player 1.

In Gensbittel \textit{et al}. \cite{Gensbittel_2014}, the authors considered a weaker notion of "a more informed player" but they made a different assumption on the transition function, especially that the less informed player has no influence on the evolution of beliefs of both players. It is not clear if our result still holds under this weaker assumption.
\end{remarque}
%
%
%
%
\subsection{Evaluation}

At stage $t$, the space of past histories of player $1$ is
$H^1_{t}=(C \times I )^{t-1}\times C$ and the space of past histories of player $2$ is
$H^2_{t}=(D \times J )^{t-1}\times D.$ Let
$H_{\infty}=(K \times C\times D \times I \times J)^{\infty}$ be the space of infinite
\emph{plays}. For any play $h=(k_s,c_s,d_s,i_s,j_s)_{s\geq 1}$, we denote by $h_t$ its projection on $H_t$, by $h_t^1$ its projection on $H_t^1$, and by $h_t^2$ its projection on $H_t^2$.



A \emph{(behavior) strategy} for player $1$ is a sequence $(\sigma_t)_{t
\geq 1}$ of functions $\sigma_t: H^1_t \rightarrow \Delta(I)$. A
\emph{(behavior) strategy} for player $2$ is a sequence $\tau=(\tau_t)_{t
\geq 1}$ of functions $\tau_t: H^2_t \rightarrow \Delta(J)$. We denote
by $\Sigma$ and $\T$ players' respective sets of strategies. An initial distribution $\pi\in\De(K\times C\times D)$ and a couple of strategies $(\sigma,\tau)$ define a probability distribution over the set of infinite plays, which we denote by $\PP^\pi_{\sigma,\tau}$. Let $\EE^\pi_{\sigma,\tau}$ be the expectation w.r.t. to $\PP^\pi_{\sigma,\tau}$. 

{\color{vert}For any given $\pi\in\De(K\times C\times D)$, let $\gamma_n(\pi,\sigma,\tau)$ \big(resp. $\gamma_\lambda(\pi,\sigma,\tau)$\big) be the expected $n$-stage payoff \big(resp. $\lambda$-discounted payoff\big) associated with $(\sigma,\tau)\in\Sigma\times \T$. We denoted by $v_n(\pi)$ the $n$-stage value and by $v_\lambda(\pi)$ the $\lambda$-discounted value.} 

\begin{definition}\label{asymptotic} Given an initial distribution $\pi \in\Delta(K\times C\times D)$,  the game $\Gamma(\pi)$ has an asymptotic value $v(\pi)$ if:
\[
v(\pi)=\lim_{n\to\infty} v_n(\pi)=\lim_{\lambda\to 0} v_\lambda(\pi).
\]
\end{definition} 

\begin{definition}\label{sign_maxmin}
Given an initial distribution $\pi \in\Delta(K\times C\times D)$, the game $\Gamma(\pi)$ has a \emph{uniform maxmin} $\underline{v}_\infty(\pi)$ if:
\begin{itemize}
\item Player $1$ can \emph{guarantee} $\underline{v}_\infty(\pi)$, $i.e.$, for all $\varepsilon >0$ there exists a strategy $\sigma^*\in \Sigma$ of player $1$ and $n_0 \geq 1$ such that
\begin{align*}
\forall n\geq n_0, \ \forall \tau\in \T,\ \gamma_n(\pi,\sigma^*,\tau) \geq \underline{v}_\infty(\pi)-\varepsilon.
\end{align*}
\item Player $2$ can \emph{defend} $\underline{v}_\infty(\pi)$, $i.e.$, for all $\varepsilon >0$ and for every strategy $\sigma \in \Sigma$ of player $1$, there exists $n_0\geq 1$ and $\tau^*\in \T$ such that
\begin{align*}
 \forall n\geq n_0,\ \gamma_n(\pi,\sigma,\tau^*) \leq  \underline{v}_\infty(\pi) +\varepsilon.
\end{align*}
\end{itemize}
\end{definition}

The game $\Gamma(\pi)$ has a \emph{uniform minmax} $\overline{v}_\infty(\pi)$ is defined similarly if  player $2$ can guarantee $\overline{v}_\infty(\pi)$ and player $1$ can defend $\overline{v}_\infty(\pi)$.

\begin{definition}\label{def:uniform} Given an initial distribution $\pi \in\Delta(K\times C\times D)$, we say that $\Gamma(\pi)$ has a uniform value if both $\bar{v}_\infty(\pi)$ and $\underline{v}_\infty(\pi)$ exist and are equal. Whenever the uniform value exists, we denote it by $v_\infty(\pi)$.
\end{definition}

\subsection{Results}

\begin{theoreme}\label{thm:maxmin}
Let $\Gamma$ be a recursive game such that player $1$ is more informed than player $2$. Then for every distribution $\pi \in \Delta^1(K \times C \times D)$, {\color{vert}both the asymptotic value and the uniform $maxmin$ exist and are equal}:
\[\underline{v}_\infty(\pi) =\lim v_n(\pi)=\lim v_\lambda(\pi)\]
\end{theoreme}

\noindent By symmetry, we deduce a similar result by exchanging the roles of player $1$ and player $2$. When the information is symmetric, both results are true and we obtain the existence of the uniform value.

\begin{corollary}
Let $\Gamma$ be a  recursive game with symmetric signals. Then for every $\pi \in \Delta^*(K \times C\times D)$, the game $\Gamma(\pi)$ has a uniform value.
\end{corollary}
It is known from Ziliotto \cite{Ziliotto_2013} that stochastic games with symmetric signals may have no uniform value. Therefore recursive games have very particular properties. It is a challenging task to identify the subclass of {\color{vert} repeated games} with $\underline{v}_\infty(\pi)=\lim v_n(\pi)=\lim v_\lambda(\pi)$. 

{\color{rouge}
\begin{remarque} \  Note that we have assumed that the stage payoff on absorbing states does not depend on the actions played. Under this assumption, players' strategies have only an influence on non-absorbing plays. Therefore, without loss of generality, we assume in the following that players observe whenever an absorption occurs and in which state it is. 

If we consider that the payoff in absorbing states still depends on the actions played, then our proof does not work. Indeed the auxiliary game introduced in {\color{violet2} Proposition \ref{prop:reduced} is not recursive anymore}. The result $\underline{v}_\infty(\pi)=\lim v_n(\pi)=\lim v_\lambda(\pi)$ is unknown for this general case.
\end{remarque}} 

%
%
%
{\color{jaune}
\begin{remarque}   It is not known whether recursive games with any structure of signals have a uniform value. As highlighted in Rosenberg and Vieille \cite{Rosenberg_2000}, the equicontinuity of the $\lambda$-discounted value functions is sufficient in order to deduce the existence of the uniform value  for recursive games (with perfect observations). For a recursive game with any structure of signals, one can introduce the game associated with a universal belief space but we do not know a metric on this space such that the $\lambda$-discounted values or the $n$-stage values are equicontinuous/totally bounded.
\end{remarque}
}

\subsection{Proof of Theorem \ref{thm:maxmin}}

{\color{vert}

{\color{vert}We introduce some notations concerning different belief hierarchies}. Denote by $B_1=\Delta(K)$ the set of beliefs of player $1$ on the state variable. Denote by $B_2=\De_f(B_1)=\De_f(\Delta(K))$ the set of beliefs of player $2$ on the {\color{vert}(first-order)} beliefs of player $1$. Finally, we denote by $\De_f(B_2)=\De_f(\De_f(\Delta(K)))$ the set of probability distibutions over the second-order beliefs of player $2$.\\

\noindent \textbf{Overview of the proof} \ 

{\color{jaune} We fix $\Gamma$ a recursive game with signals such that player $1$ is more informed than player $2$}. The first subsection presents general properties for repeated games with one player more informed than the other. Given any $\pi \in\Delta^1(K\times C\times D)$,  we can define   {\color{jaune} the distribution} of the beliefs of player $2$ on the beliefs of player $1$ about the {\color{rouge}{state}}.  This defines a function $\Phi$ from $\Delta^1(K\times C\times D)$ to $\De_f(B_2)$.  Applying results in Gensbittel \textit{et al}. \cite{Gensbittel_2014}, we know that  $v_n(\pi)$ depends on $\pi$ only through $\Phi(\pi)$. This enables us to show that the value function $v_n$, defined on $\Delta^1(K\times C\times D)$, induces a canonical function $\hat{v}_n$ defined on $B_2$ such that $v_n(\pi)=\hat{v}_n\big(\Phi(\pi)\big)$ and the family $\{\hat{v}_n, n\geq 1\}$ is totally bounded.

In the second subsection,  we introduce an auxiliary recursive game $\mathcal{G}$ which is defined on $B_2$ and is played with pure actions. We prove  in Proposition \ref{equal} that the $n$-stage value of $\mathcal{G}_n$ is equal to $\hat{v}_n$. Therefore, $\mathcal{G}$ satisfies the conditions of Corollary \ref{theo3} and it has a uniform value $w_\infty$.  {\color{rouge} It follows that $\Gamma(\pi)$ has an asymptotic value equal to $w_\infty\big(\Phi(\pi)\big)$}.


The third subsection proves that (cf.  Proposition \ref{prop:p1-guarantee-Gamma}) player 1 can uniformly guarantee $w_\infty\big(\Phi(\pi)\big)$ in $\Gamma(\pi)$ by mimicking uniform $\varepsilon$-optimal strategies in $\mathcal{G}\big(\Phi(\pi)\big)$. 

The last subsection proves that  (cf. Proposition \ref{prop: p2-defend-Gamma}) that player $2$ can uniformly defend $w_\infty\big(\Phi(\pi)\big)$ by introducing a second auxiliary recursive game $\mathcal{R}$}.

\vspace{3mm}

\subsubsection{Canonical value function $\hat{v}_n$}

{\color{vert} We follow  in this subsection Gensbittel \textit{et al}. \cite{Gensbittel_2014} to introduce the canonical function $\hat{v}_n$. Note that to obtain results in this subsection,  the additional assumption that player 1 controls the transition  (made later in their paper) is not used in Gensbittel \textit{et al}. \cite{Gensbittel_2014}.} \\


{\color{vert}For convenience,  we extend the definition of $\Gamma(\pi)$ to a larger family of initial probability distributions.} Given any two finite sets $C'$ and $D'$ and $\pi\in \Delta^1(K\times C' \times D')$, $\Gamma(\pi)$ is the game where $(k,c',d')$ is drawn at stage $1$ according to $\pi$, player $1$ observes $c'$, player $2$ observes $d'$ (which is contained in $c'$ $\pi-a.s.$) and then from stage $2$ on, the game is played as previously described with signals in $C$ and $D$.
%

For any random variable $\xi$ defined on a probability space $(\Omega,\A,\PP)$ and $\F$ a sub $\sigma$-algebra of $\A$, let $\CL_{\PP}(\xi\mid \F)$ denote the conditional distribution of $\xi$ given $\F$, which is seen as a $\F$-measurable random variable\footnote{All random variables appearing here take only finitely many values so that the definition of conditional laws does not require any additional care about measurability.} and let $\CL_{\PP}(\xi)$ denote the distribution of $\xi$.

\begin{notation} For every strategy profile $(\si,\tau)\in \Sigma\times \T$, we denote the first-order belief of player $1$ on $K$ at stage $n$ given $h^1_n$ by $p_n\in B_1$, the second-order belief of player $2$, $i.e.$, his belief about the belief of player $1$ on $K$ at stage $n$ given $h^2_n$ by $x_n\in B_2$, and the distribution of $x_n$ by $\eta_n \in \De_f(B_2)$, $i.e.$,
$$p_{n} \triangleq  \CL_{\PP_{\si\tau}^{\pi}} (k_{n}\mid h^1_n), \quad  x_n\triangleq  \CL_{\PP^{\pi}_{\sigma\tau}}(p_n | h_{n}^{2}), \quad \text{and}\quad \eta_n\triangleq  \CL_{\PP^{\pi}_{\sigma\tau}}(x_n).$$ 
\end{notation}

\begin{notation}
\color{vert}{For any ${\pi} \in\De^1(K \times C' \times D')$  where $C'$ and $D'$ are two finite sets, the \emph{image} of $\pi$ is given by the following function in $\De_f(B_2)$}:
\begin{align*}
 \Phi({\pi})  &\triangleq  \CL_{\pi}\left( \CL_{\pi} \left( \CL_{\pi} ( k_1 | c_1 ) |d_1 \right)  \right), \\
                   & =\sum_{d\in D'} \pi(d) \delta_{\left( \sum_{c\in C'} \pi(c|d) \delta_{\pi(.|c,d)} \right)}.
\end{align*}
\end{notation}
{\color{orange2} {\color{vert} The interpretation of $\Phi(\pi)$ is as follows:} with probability $\pi(d)$, player $2$ observes the signal $d$ and {\color{vert}believes} that: player $1$ received the signal $c$ with probability $\pi(c|d)$ {\color{vert} and therefore  player $1$'s belief over $K$ } is $\pi(.|c,d)$.}

The assumptions imply that if $\pi \in \Delta^1(K\times C' \times D')$, then $\pi$ satisfies the following two properties:\\

P1) ${\pi}(c){\pi}(k,c,d)={\pi}(k,c){\pi}(c,d),\ \forall (k,c,d)\in K \times C'\times D'$.\\

P2) There exists a map $f_1=f_1^{\pi}:  C' \to B_2$ such that $x_1= f_1(c_1), \ {\pi} \text{-almost surely}$.\\

{\color{bleu}
Under P1) and P2), Proposition 1 of Gensbittel \textit{et al}. \cite{Gensbittel_2014} applies and we  obtain the following result, which states that {\color{vert}the value of any $n$-stage game depends on any initial distribution $\pi$ only through its  image $\Phi(\pi)$}}.

{\color{vert}\begin{proposition}\  [Gensbittel et al. 2014] \ Let $C'$ and $D'$ be two finite sets. Let $\pi,\pi'\in \De^1(K\times C'\times D')$ and let $n\geq 1$. If $\Phi(\pi)=\Phi(\pi)$, then $v_n(\pi)=v_n(\pi')$. 
\end{proposition}}


Reciprocally, given $\eta \in \De_f(B_2)$, let us construct a \textit{canonical distribution} $\pi$ satisfying $\Phi(\pi)=\eta$.\\

\noindent \textbf{The canonical game $\hat{\Gamma}(\eta)$.} \ Given $\eta \in \De_f(B_2)$. Define two finite sets $D':=supp(\eta)\subseteq  B_2$ and $C':=D'\times \left(\bigcup_{x\in supp(\eta)}supp(x)\right)$, and a probability distribution $\pi(\eta) \in \Delta(K\times C' \times D')$ by
\[ \forall (k,p)\in K \times B_1,\ x,x' \in B_2, \, {\pi}(k,(p,x),x'):=\begin{cases} \eta(x)x(p)p(k) \text{ if } x=x'  \\ \phantom{\eta(x)x(} 0 \phantom{)p(k)}{} \text{ if } x \neq x'. \end{cases}\]
By construction, $\pi(\eta)$ can be seen as an element of $\De^1(K\times C'\times D')$, and satisfies $\Phi(\pi(\eta))=\eta$. The \textit{canonical game} of $\Gamma(\pi)$ is denoted as $\hat{\Gamma}(\eta)$. Its value, denoted by $\hat{v}_n(\eta)$, is equal to $v_n(\pi(\eta))$ the value of $\Gamma_n(\pi(\eta))$. If $\eta=\de_x$ for some $x\in B_2$, we denote $\hat{v}_n(x)$ for $\hat{v}_n(\de_x)$.

{\color{bleu}
Informally, the game $\hat{\Gamma}(\eta)$ proceeds as follows: $\eta$ is common knowledge, player 2 is informed about the realization $x$ of a random variable with law $\eta$ (player 2 learns his beliefs). Then player 1 is informed about $x$ (his opponent's beliefs) and about the realization $p$ of a random variable with law $x$ (his own beliefs). The state variable is finally chosen according to $p$, but no player observes it.
}


By the above construction, one obtains that: $v_n(\pi)=\hat{v}_n\big(\Phi(\pi)\big)$ for any $\pi\in\De^1(K\times C'\times D')$. \\ 

{\color{jaune}The result below follows from Proposition 2 of Gensbittel \textit{et al}. \cite{Gensbittel_2014}. The Wasserstein metric $\mathbf{d}$ on $B_2=\De\big(B_1\big)$ is defined by:
\[
\forall x,y \in B_2,\  \mathbf{d}(x,y)=\sup_{f \in \mathcal{D_1}} \left| \int_{B_1}f(p)x(dp) - \int_{B_1}f(p)y(dp) \right|,
\]}
where $\mathcal{D_1}$ is the set of $1$-Lipschitz function from $(B_1,\|.\|_1)$ to $[-1,1].$

\begin{proposition}\label{canlip}[Gensbittel et al. 2014]
Let $\eta \in \De_f(B_2)$, $n\geq 1$ and  let $x \in  B_2$. Then,
$\hat{v}_{n}(\eta)$ is linear on $\De_f(B_2)$ and, as a mapping on $B_2$, $\hat{v}_{n}(x)$ is $1$-Lipschitz for the Wasserstein metric $\mathbf{d}$.
\end{proposition}

Since the state space $B_2$ is totally bounded for the Wasserstein metric, we deduce by Arzela-Ascoli theorem that the set of functions $\{\hat{v}_n, n\geq 1\}$ is totally bounded.

%

\subsubsection{Auxiliary recursive game and asymptotic value}

\noindent Let $\G=(X,A,B,G,\ell)$ be the stochastic game {\color{vert}played in pure strategies}, defined by:
\begin{itemize}
\item the state space $X=\De_f(\De(K))$ (endowed with the Wasserstein metric $\mathbf{d}$).
\item the action space $A=\{f:\De(K)\to \Delta(I)\}$ and for all $x\in X$, $A(x)=\{\mathrm{supp}(x) \to \Delta(I)\}$ for player $1$.
\item the action space $B=\De(J)$ for player $2$. 
\item the payoff function $G:X \rightarrow [-1,1]$, defined for any $x\in X$ by
$G(x):= \sum_{p \in \Delta(X)} g(p)x(p)$. 
\item the transition function  $\ell: X \times A \times B \rightarrow \De_f(X)$
defined as
\textcolor{bleu}{$\ell(x,a,b):=\Phi(Q(x,a,b))$.}
Here, $Q(x,a,b)\in \Delta_f(K \times (\De(K)\times C) \times D)$ 
 is the joint distribution of $(k_2,(p,c_2),d_2)$ in the canonical game $\widehat{\Gamma}(\de_x)$ when the players play 
 $(\si_1,\tau_1)=(a,b)$ at stage $1$. The sets $K, C,D$ and $\mathrm{supp}(x)$ being finite, $Q$ can be seen
 as an element in $\Delta^1(K \times C' \times D')$ with $C'$ a finite subset of $\Delta(K) \times C$ and $D'= D$.
\end{itemize}

For any $x\in X$, we denote by $\mathcal{G}(x)$ the game starting at $x$. We extend the definition to $\mathcal{G}(z)$ for any  $z\in \Delta_f(X)$ such that the initial state is chosen randomly along $z$.\\
%
%


Since players observe when and where absorption occurs, their beliefs (first and second-order) are either supported on $K^0$ \big(therefore respectively in $\De(K^0)$ and in $\De(\De(K^0))$\big) or supported on each single point $k\in K^*$ (to be $\de_k$ and to be $\de_{\de_k}$). 

\begin{proposition}\label{prop:reduced}
Let $X_r=\Delta_f(\Delta(K^0)) \bigcup \{\de_{\de_k}: k\in K^*\}$. The game $\G^r=(X_r,A,B,G,\ell)$ with the state space  $X_r$ is well defined and is recursive with the absorbing states $\{\de_{\de_k}:k\in K^*\} $.
\end{proposition}


%
{\color{rouge} In the following, we identify each $\de_{\de_k}$ with $k$ itself for any $k\in K^*$, and write $X_r=\Delta_f(\Delta(K^0)) \cup K^*$.} {\color{jaune}  By abuse of notations, we write again $X$ for $X_r$ and $\G$ for $\G^r$.}  


{
\color{vert}

\begin{proposition}\label{equal}
 For every $n \geq 1$, the $n$-stage game $\G_n$ has a value $w_n$ in pure strategies. Moreover, for every $x\in X$, $w_n(x)=\hat{v}_n(x)$.
\end{proposition}

\begin{proof}
{\color{violet2} We prove the result by induction on $n\geq 1$. Let $n=1$. Given $x\in X$, the game $\G_1(x)$ has a value $w_1(x)$ and it is equal to $w_1(x)=G(x)=\sum_{p} g(p)x(p)$. It is equal to $\hat{v}_1(x)$ by construction. 
This initializes our induction.
}
Let $n\geq 1$ such that $w_n$, the value of $\G_n$, exists in pure strategies, and  for every $x\in X$, $w_n(x)=\hat{v}_n(x)$. Gensbittel \textit{et al}. \cite{Gensbittel_2014} showed in the proof of their Proposition $5$ that the family $\{\hat{v}_n, n\geq 1\}$ satisfies the Shapley equation: for every $x\in X$ and for every $n\geq 1$,
\begin{align*}
\hat{v}_{n+1}(x) & = \sup_{a \in A(x)} \inf_{b \in B } \EE_{\ell(x,a,b)} \left[ \frac{1}{n+1} g(x,a,b) +\frac{n}{n+1} \hat{v}_{n}(x') \right]\\
                                 & = \inf_{b \in B }  \sup_{a \in A(x)} \EE_{\ell(x,a,b)} \left[ \frac{1}{n+1} g(x,a,b) +\frac{n}{n+1} \hat{v}_{n}(x') \right],
\end{align*}
where the random variable $x'\in X$ is chosen along the law $\ell(x, a,b)(\cdot)$. By the inductive assumption, we can replace $\hat{v}_n$ by $w_n$ on the right hand side of above equation, to obtain that:
\begin{align*}
\forall x\in X,\ \  \hat{v}_{n+1}(x)= & \sup_{a \in A(x)} \inf_{b \in B } \EE_{\ell(x,a,b)} \left[ \frac{1}{n+1} g(x) +\frac{n}{n+1} w_{n}(x') \right] \\
 =& \inf_{b \in B }  \sup_{a \in A(x)} \EE_{\ell(x,a,b)} \left[ \frac{1}{n+1} g(x) +\frac{n}{n+1} w_{n}(x') \right].
\end{align*}
We now use the above equation to show that both players can guarantee $\hat{v}_{n+1}(x)$ in $\G_{n+1}(x)$ in pure strategies. Let $x\in X$ be fixed and $a^*$ be an action of player $1$ such that
\begin{align}\label{bouh2}
\inf_{b \in B } \EE_{\ell(x,a^*,b)} \left[ \frac{1}{n+1} g(x) +\frac{n}{n+1} w_{n}(x') \right] \geq \hat{v}_{n+1}(x).
\end{align}
Again by inductive assumption, let $\sigma^*_n(x')$ be an optimal pure strategy in $\G_n(x'),\forall x'\in X$. We define the strategy $\sigma^*_{n+1}(x)$ to play $a^*$ at the first stage and then $\sigma^*_{n}(x')$ where $x'$ is the current state at stage $2$. $\sigma^*_{n+1}(x)$ is pure and guarantees player 1 the payoff in $\G_{n+1}(x)$ no smaller than the left hand side of Equation (\ref{bouh2}), hence $\hat{v}_{n+1}(x)$. {\color{jaune} A similar construction for player 2 finishes the inductive proof.}  \end{proof} } \\ 
 
Therefore, {\color{vert}by Proposition \ref{canlip} the family of $n$-stage values $\{w_n\}$ is totally bounded for the uniform norm}, and we can apply Corollary $\ref{theo3}$ (with infinite sets of actions) for the game $\G$. 
%

\begin{proposition}\label{aux_uniforme}
For every $z\in \Delta_f(X)$, the game $\G(z)$ has a uniform value denoted by $w^*_{\infty}(z)$. Moreover both players can uniformly guarantee the value with pure strategies that depend on the history of states but not on the past actions.
\end{proposition}

{\color{vert} Since $v_n(\pi)=\hat{v}_n\big(\Phi(\pi)\big)=w_n\big(\Phi(\pi)\big)$, and the same construction of the canonical value function $\hat{v}_\lambda$ implies $v_\lambda(\pi)=\hat{v}_\lambda\big(\Phi(\pi)\big)=w_\lambda\big(\Phi(\pi)\big)$, we deduce the existence of the asymptotic value in $\Gamma(\pi)$. in the game $\Gamma(\pi)$ for every $\pi \in \Delta^1(K\times C \times D)$.}

\begin{proposition}\label{prop:asymptotic}For every $\pi \in \Delta^1(K\times C\times D)$, we have  \[
\lim_{n\to\infty} v_n(\pi)=\lim_{\lambda\to0}v_\lambda(\pi)=w^*_\infty(\Phi(\pi)).
\]
\end{proposition}

%
%

\subsubsection{Player $1$ uniformly guarantees $w_\infty^*$}

We first show that player 1 is able to compute in the original game $(p_t)_{t\geq 1}$ his first-order beliefs and {\color{vert} $(x_t)_{t\geq 1}$ the second-order beliefs of player 2} without knowing the strategy of player 2.

\begin{lemme} \label{lem:p1-calculate-belief} Let $(\sigma,\tau)$ be a pair of strategies in $\Gamma(\pi)$. For every $t\geq 1$, $p_t=\mathcal{L}_{\PP_{\sigma,\tau}^\pi}(k_t|h_t^1)$ and $x_t=\mathcal{L}_{\PP_{\sigma,\tau}^\pi}(p_t|h_t^2)$ are independent of $\tau$ for all $h_t^1, h_t^2$.
\end{lemme}

\begin{proof} 
{\color{violet2}
Let $(\sigma,\tau)$ be a pair of strategies and $\pi \in \Delta(K\times C\times D)$, we write $\PP:=\PP^\pi_{\sigma,\tau}$  for short. 
Let $h=(k_s,c_s,d_s,i_s,j_s)_{s\geq 1} \in H_\infty$. For any $t\geq 1$, we define \[\beta(h_t) \equal \pi(k_1,c_1,d_1) \prod\limits_{\ell=1}^{t-1} q(k_\ell,i_\ell,j_\ell )(k_{\ell+1},c_{\ell+1},d_{\ell+1})\] with the convention $\beta(k_1,c_1,d_1)=\pi(k_1,c_1,d_1)$. These notations help to write
\[
\PP (h_t)= \beta(h_t)\prod_{\ell=1}^{t-1} \sigma_t(h^1_\ell)[i_\ell] \tau_\ell(h^2_\ell)[j_\ell].
\]
The key point is that under $\PP(\cdot)$, $i_{t-1}$, $j_{t-1}$ and $d_t$ are $c_t$-measurable whereas $j_{t-1}$ is $d_t$-measurable. It follows that after observing $(c_1,...,c_t)$, player $1$'s belief is:
\begin{align*}
p_t(k_t)& = \PP \left(k_{t}|c_1,...,c_{t} \right) = \PP \left(k_{t}| c_1,d_1,i_1,j_1,....,c_t,d_t \right)\\
& =\frac{\sum_{k'_1,..., k'_{t-1}} \PP(k'_1,c_1,d_1,i_1,j_1...,k_t,c_t,d_t)}{\sum_{k'_1,..., k'_{t-1},k'_t} \PP(k'_1,c_1,d_1,i_1,...,k'_t,c_t,d_t)} = \frac{\sum_{k'_1,...,k'_{t-1}} \beta(k'_1,c_1,d_1,i_1,j_1,...,k_t,c_t,d_t)}{ \sum_{k'_1,..., k'_{t-1},k_t} \beta(k'_1,c_1,d_1,i_1,j_1,...,k'_t,c_t,d_t)},\end{align*}
which depends on neither $\sigma$ nor $\tau$. We now consider $x_{t}=\mathcal{L}_{\PP}\big(p_{t}|d_1,...d_t\big)$ for a given observed history $h_{t}^2=(d_1,...,d_t)$ of player $2$, which is decomposed as: 
\[
x_{t}=\mathcal{L}_{\PP}\Big(\mathcal{L}_{\PP}(k_{t+1}|c_1,...,c_t)|d_1,..,d_t\Big)=\sum_{c'_1,...,c'_t} \PP(c'_1,...,c'_t|d_1,...,d_t)\de_{\mathcal{L}_{\PP}(k_{t+1}|c'_1,...,c'_t)}.
\]
By the previous result that $\mathcal{L}_{\PP}(k_{t+1}|c'_1,...,c'_t)$ does not depend on $\tau$, it is sufficient to prove that $\PP(c'_1,...,c'_t|d_1,...,d_t)$ is independent of $\tau$. {\color{jaune} Let us consider a sequence of signals $(c'_1,...,c'_t)$ inducing $(d_1,...,d_t)$ that we complete with  $(i'_1,...,i'_t)$ the sequence of actions it contains. This gives us}
\begin{align*}
\PP(c'_1,...,c'_t|d_1,...,d_t)&=\PP(c'_1,d_1,i'_1,j_1,...,c'_t,d_t|d_1,j_1,...,d_t) =\frac{ \PP(c'_1,d_1,i'_1,j_1,...,c'_t,d_t)}{ \PP(d_1,j_1,...,d_t)}\\
                                   &= \frac{\sum_{k'_1,...,k'_t}  \beta(h'_t) \prod_{\ell=1}^{t-1}\sigma_\ell(h^{'1}_\ell)[i'_\ell]}{\sum_{k'_1,...,k'_t} \sum_{c'_1,...,c'_t} \beta(h'_t) \prod_{\ell=1}^{t-1} \sigma_\ell(h^{'1}_\ell)[i'_\ell]},
\end{align*}
where $h'_\ell=(k'_1, c'_1,d_1,i'_1,j_1,...,k'_\ell, c'_\ell,d_\ell)$ is the history of stage $\ell$ and $h^{'1}_\ell=(c'_1,i'_1,...,c'_\ell)$ is the private history of player 1 of stage $\ell$. The right hand side of the above equation does not depend on the strategy of player $2$ {\color{jaune} and the result is obtained}.}
\end{proof}

\vspace{4mm}

Before building the strategy of player $1$, we prove that the transition rule $\ell(\cdot):X\times A\times B \to \De_f(X)$ of the auxiliary game is linear with respect to $b\in B$ (the action of player $2$).
 
\begin{lemme}\label{lem:ell-b} For any $(x,a)\in X\times A$ and $\overline{b}=\sum_{s\in S}\lambda_s b_s$ a convex combination in $B$, we have $$\ell(x,a, \overline{b}) =\sum_{s\in S} \lambda_s \ell \left(x,a,b_s \right).$$
\end{lemme}

\begin{proof} Let $(x,a)\in X\times A$ and $b\in B$. Recall that $Q:=Q(x,a,b)$ denotes a distribution in $\De_f(K\times(\De(K)\times C)\times D)$, which can be seen as as an element in $\De^1(K\times C'\times D')$ with $C'=supp(x)\times C$ a finite subset of $\De(K)$ and $D'=D$ We have by definition of the image mapping $\Phi(\cdot)$:
$\ell(x,a,b)=\Phi(Q)=\sum_{d'\in D'} Q(d')\de_{\mathcal{L}_{Q}\left(\mathcal{L}_{Q}(k|c')|d'\right)}.$ 

Similarly to the previous lemma, for every $(c',d')=((p,c),d')\in C' \times D'$, $\mathcal{L}_{Q}\Big(\mathcal{L}_{Q}\big(k|(p,c)\big)|d'\Big)$ does not depend on $b$. Indeed, the signal $(c',d')$ contains the action $(i_1,j_1)=\left(\hat{\imath}(c'),\hat{\jmath}(d')\right)$ and $c'$ contains $d'$ $a.s$. It follows that
\[
\PP_{Q}(k|p,c)=\frac{a(p)[\hat{\imath}(c)]b[\hat{\jmath}(c)]q^{K\times C}\big(p,\hat{\imath}(c),\hat{\jmath}(c)\big)(k,c')}{a(p)[\hat{\imath}(c)]b[\hat{\jmath}(c)]q^C\big(p,\hat{\imath}(c),\hat{\jmath}(c)\big)(c')}=\frac{q^{K\times C}\big(p,\hat{\imath}(c),\hat{\jmath}(c)\big)(k,c')}{q^C\big(p,\hat{\imath}(c),\hat{\jmath}(c)\big)(c')}
\]
and
\[
\PP_{Q}(p,c|d')=\frac{x(p) q^{C}\big(p,a(p),\hat{\jmath}(d')\big)(c)}{\sum_{p\in supp(x)} x(p) q^{D}\big(p,a(p),\hat{\jmath}(d')\big)(d')}.
\]%
Since these quantities do not depend on $b$, we will not precise $b$ in the following. The application $Q(x,a,b)$ being linear in $b$, we can easily deduce the announced result:
 \begin{align*}
\Phi\left(Q\left(x,a,\overline{b}\right)\right)& = \sum_{d'\in D'} Q\left(x,a,\overline{b}\right)(d') \de_{\mathcal{L}_{Q(x,a,.)}\left(\mathcal{L}_{Q(x,a,.)}(k|c')|d'\right)} \\
& = \sum_{d'\in D'} \left( \sum_{s\in S} \lambda_s Q(x,a,b_s)(d')\de_{\mathcal{L}_{Q(x,a,.)}\left(\mathcal{L}_{Q(x,a,.)}(k|c')|d'\right)}\right)\\
&=\sum_{s\in S} \lambda_s \Phi(Q(x,a,b_s).
\end{align*}
\end{proof}  

We now use the two previous lemmas to prove that  player $1$ uniformly guarantees $w^*_\infty\left(\Phi(\pi)\right)$ in the game $\Gamma(\pi)$.

\begin{proposition}\label{prop:p1-guarantee-Gamma} Player 1 uniformly guarantees $w^*_\infty\left(\Phi(\pi)\right)$ in $\Gamma(\pi)$.
\end{proposition}

\begin{proof} 
Fix any $\varepsilon>0$. We divide the proof into three steps. {\color{jaune} First, we define the optimal strategy $\hat{\sigma}$ in $\Gamma(\pi)$. Then we show how to link the distribution over the states in $\G$ to the distribution of second-order beliefs in $\Gamma$}. Finally, we deduce that the strategy $\hat{\sigma}$ is uniform $\varepsilon$-optimal.\\

\noindent \underline{Step I: Defining the strategy.}\\
 
Consider the auxiliary game $\G(z)$ with $z=\Phi(\pi)\in\De_f(X)$. According to Proposition \ref{aux_uniforme}, player 1 has pure uniform $\varepsilon$-optimal strategies which depend on histories only through the states but not the actions. {\color{violet2} With a slight abuse of notations,} there exists $\hat{\sigma}^*:\bigcup_{t=1}^\infty X^t\to A=\{a:\De(K)\to\De(I)\}$ and $N_0\geq 1$ such that
\[\hat{\gamma}_n(z,\hat{\sigma}^*,\hat{\tau})\geq w^*_\infty(z)-\varepsilon \text{ for all $n\geq N_0$ and for all } \hat{\tau}:\bigcup_{t=1}^\infty X^t\to B=\De(J)\]
where $\hat{\gamma}_n(z,\hat{\sigma}^*,\hat{\tau})$ is the expected $n$-stage average payoff in the auxiliary game $\G(z)$ induced by $(z,\hat{\sigma}^*, \hat{\tau})$.

We define the strategy $\sigma^*\in \Sigma$ in the game $\Gamma(\pi)$ such that for any $h_t^1$,
\[
\sigma^*(h^1_t)=\hat{\sigma}^*(x_1,...,x_t)[p_t]
\text{\ with\ } p_t=\mathcal{L}_{\PP^\pi_{\sigma^*}}(k_t|h_t^1) \text{ and } x_t=\mathcal{L}_{\PP^\pi_{\sigma^*}}(p_t|h_t^2).\] 
By Lemma \ref{lem:p1-calculate-belief}, this is a well defined strategy of player $1$ since he can compute $p_t$ and $x_t$ at every stage $t\geq 1$. We now check that the strategy $\sigma^*$ uniformly guarantees $w_\infty^*(z)-\varepsilon$ in $\Gamma(\pi)$. \\

\noindent \underline{Step II: Linking the probability law of beliefs}\\

{\color{jaune} Let $\tau \in \T$ be a strategy in $\Gamma(\pi)$. We define a strategy $\hat{\tau}$ in $\G\left(\Phi(\pi)\right)$ such that $(\pi,\sigma^*,\tau)$ and $(\Phi(\pi),\hat{\sigma},\hat{\tau})$ generate the same probability law for $(x_1,...,x_t,...)$.} 
With a slight abuse in notation, we denote by $\hat{\tau}$  the strategy in $\G$  such that for all $(x_1,...,x_t)\in X^t$,
\[
\hat{\tau}(x_1,...,x_t)=\sum_{h_t^2\in H_t^2(x_1,...,x_t)} \PP^\pi_{\sigma^*,\tau}(h_t^2|x_1,...,x_t)\tau(h_t^2),
\]
where $H_t^2(x_1,...,x_t)=\{h_t^2\in H_t^2|\mathcal{L}_{\PP^\pi_{\sigma^*,\tau}}(k_l|h_\ell^2)=x_\ell,\ 1\leq  \ell \leq t\}$
denotes the set of player 2's $t$-stage histories in $\Gamma$ that induce the beliefs $(x_1,...,x_t)$. 

\begin{lemme}\label{lem:law-belief} Let $\sigma^*$ and $\hat{\tau}$ be constructed as above given $\hat{\sigma}^*$ and $\tau$, we have:  
\[\forall t\geq 1, \ \mathcal{L}_{\PP^\pi_{\sigma^*,\tau}}(x_1,...,x_t)=\mathcal{L}_{\PP^z_{\hat{\sigma}^*,\hat{\tau}}}(x_1,...,x_t).
\]
\end{lemme}

\noindent \textbf{Proof of Lemma \ref{lem:law-belief}:} \  We prove the lemma by induction on $t\geq 1$. For $t=1$, the law of $x_1$ is independent of the strategy profile. By definition of the image mapping $\Phi(\cdot)$,
\[
\mathcal{L}_{\PP^\pi_{\sigma^*,\tau}}(x_1)=\mathcal{L}_\pi\left(\mathcal{L}_{\pi}\big(\mathcal{L}_\pi(k_1|c_1)|d_1\big)\right)=\Phi(\pi).
\]
As $\Phi(\pi)=z$, the probability law to choose the initial state $x_1\in X$ in $\G(z)$,   $\mathcal{L}_{\PP^z_{\hat{\sigma}^*,\hat{\tau}}}(x_1)=\Phi(\pi)$.

Suppose now that we have proved that $\mathcal{L}_{\PP^\pi_{\sigma^*,\tau}}(x_1,...,x_t)=\mathcal{L}_{\PP^z_{\hat{\sigma}^*,\hat{\tau}}}(x_1,...,x_t)$ for some $t\geq 1$. It is then sufficient to prove that conditional on any realization\footnote{For this part of the proof it is convenient to differentiate the random variable describing the second order belief (or the state in $\G$) that will be denoted by $x_t$ from its realization denoted by $\tilde{x}_t$.} $\tilde{s}_t:=(\tilde{x}_1,...,\tilde{x}_t)\in \left(B_2\right)^t$, $$\mathcal{L}_{\PP^\pi_{\sigma^*,\tau}}(x_{t+1}|x_1=\tilde{x}_1,...,x_t=\tilde{x}_t)=\mathcal{L}_{\PP^z_{\hat{\sigma}^*,\hat{\tau}}}(x_{t+1}|x_1=\tilde{x}_1,...,x_t=\tilde{x}_t).$$ 
{\color{jaune}
\noindent Fix some $\tilde{s}_t=(\tilde{x}_1,...,\tilde{x}_t)\in X^t$. By definition of  $\hat{\tau}$ and the linearity of $\ell$ showed in Lemma \ref{lem:ell-b}, we know that
\begin{equation}\label{eq:law-x_t-1}
\begin{aligned}
\mathcal{L}_{\PP^z_{\hat{\sigma}^*,\hat{\tau}}}(x_{t+1}|x_1=\tilde{x}_1,...,x_t=\tilde{x}_t) &= \ell(\tilde{x}_t,\hat{\sigma}^*(\tilde{s}_t), \hat{\tau}(\tilde{s}_t))\\
& = \sum_{h_t^2\in H_t^2(\tilde{s}_t)}\PP^\pi_{\sigma^*,\tau}(h_t^2|\tilde{s}_t) \ell\big(\tilde{x}_t,\hat{\sigma}^*(\tilde{s}_t), \tau(h_t^2)\big) .
\end{aligned}
\end{equation} 
By definition of the conditional expectation, we have in $\Gamma$,
\[
\mathcal{L}_{\PP^\pi_{\sigma^*,\tau}}(x_{t+1}|x_1=\tilde{x}_1,...,x_t=\tilde{x}_t) =\sum_{h_t^2\in H_t^2(\tilde{s}_t)}\PP_{\sigma^*,\tau}^\pi(h_t^2|\tilde{s}_t)\mathcal{L}_{\PP^\pi_{\sigma^*,\tau}}(x_{t+1}|h_t^2).
\]
Thus, it is sufficient to prove that for every $h_t^2\in H^2_t(\widetilde{s}_t)$, $\ell\big(\tilde{x}_t,\hat{\sigma}^*(\tilde{s}_t), \tau(h_t^2)\big)= \mathcal{L}_{\PP^\pi_{\sigma^*,\tau}}(x_{t+1}|h_t^2)$.
Let $h_t^2\in H_t^2(\tilde{x}^t)$ and $Q[h_t^2]:=Q\left(\tilde{x}_t,\hat{\sigma}^*(\tilde{s}_t), \tau(h_t^2)\right)\in\De_f\left(K\times(\De(K)\times C)\times D\right)$ the joint distribution of $\left(k_{t+1}, (p_t,c_{t+1}), d_{t+1}\right)$ in the canonical game $\hat{\Gamma}(\de_{\tilde{x}_t})$ when $\left(\hat{\sigma}^*(\tilde{s}_t),\tau(h_t^2)\right) \in A\times B$ is played. By definition of the image mapping $\Phi(\cdot)$ and $\sigma^*$, we obtain
\begin{align*}
\mathcal{L}_{\PP^\pi_{\sigma^*,\tau}}(x_{t+1}|h_t^2)&= \mathcal{L}_{Q[h_t^2]}\left(\mathcal{L}_{Q[h_t^2]}(\mathcal{L}_{Q[h_t^2]}(k_{t+1}|c_{t+1})|d_{t+1})\right) =\Phi(Q[h_t^2])=\ell\big(\tilde{x}_t,\hat{\sigma}^*(\tilde{s}_t), \tau(h_t^2)\big).
\end{align*}
$\hfill\Box$}


\noindent \underline{Step III: Conclusion of the proof}\\ 

Finally, let us compare the payoffs in both games. If $k^* \in K^*$, we have $G(k^*)=g(k^*)=\EE^\pi_{\sigma^*,\tau}[g(k_t)|x_t=k^*]$. If $x_t\in \big(\De_f\left(\De(K^0)\right)\big)$, we have $G(x_t)=0=\EE^\pi_{\sigma^*,\tau}[g(k_t)|x_t]$. It follows that for every $x_t\in X$, we have $G(x_t)=\EE^\pi_{\sigma^*,\tau}\left [g(k_t)|x_t\right]$. By taking conditional expectation, Lemma \ref{lem:law-belief} implies that $\EE^z_{\hat{\sigma}^*,\hat{\tau}}[G(x_t)]=\EE^\pi_{\sigma^*,\tau}[g(k_t)]$. Since $\hat{\sigma}^*$ is uniform $\varepsilon$-optimal in the auxiliary game $\G\left(\Phi(\pi)\right)$, we obtain
$$\gamma_n(\pi, \sigma^*, \tau) = \hat{\gamma}_n(\Phi(\pi), \hat{\sigma}^*, \hat{\tau})\geq w^*_\infty(\Phi(\pi))-\varepsilon \text{ for all }n\geq N_0.$$
Therefore, the strategy $\sigma^*$ uniformly guarantees $w^*_\infty(\Phi(\pi))-\varepsilon$ in $\Gamma(\pi)$.
\end{proof}\\

\subsubsection{Player $2$ uniformly defends $w_\infty^*$}

We now prove that player $2$ can defend $w^*_\infty(\Phi(\pi))=\lim v_n(\pi)=\lim v_\lambda(\pi)$. The situation of player $2$ is different since he is allowed to know the strategy of player $1$. In order to prove this result, we introduce another auxiliary recursive game {\color{jaune} $\mathcal{R}$}.\\

For any $n\geq 1$, let $H'_n\subseteq H_n$ be the set of $n$-stage histories such that player 1 can deduce player 2's private signals, and $H_n^0\subseteq H_n$ be the set of $n$-stage histories containing only non-absorbing states. We consider the following game  {\color{jaune} $\mathcal{R}$} where the set of states is almost the set of distribution over {\color{jaune} all finite histories}. It is defined as follows:


\begin{itemize}
\item the state space is $Z=Z_0 \bigcup K^*$ where $Z_0=\bigcup_{n\geq 1} \De(H_n^0\cap H'_n)$,
\item the action space of player $1$ is $A=\bigcup_{n\geq 1} \{f:H^1_n \rightarrow \Delta(I) \}$ and for any $\pi_n \in \Delta(H_n)$, $A(\pi_n)= \{f:H^1_n \rightarrow \Delta(I)\}$,
\item the action space of player $2$ is $B=\bigcup_{n\geq 1} \{f:H^2_n \rightarrow \Delta(J) \}$ and for any $\pi_n \in \Delta(H_n)$, $B(\pi_n)= \{f:H^2_n \rightarrow \Delta(J)\}$,
\item the transition $Q:Z\times A \times B \rightarrow \Delta_f(Z)$ is given by:
\[
\forall  (k^*,a,b) \in K^*\times A\times B,\  Q(k^*,a,b)=\delta_{k^*},
\]
and 
\[ { \color{jaune} \forall  (z,a,b)\in Z_0\times A \times B},\ 
Q(z,a,b)= {\color{jaune}Q^0(z,a,b)} \delta_{\pi^0}+ \sum_{k\in K^*} Q(z,a,b)(k^*)\delta_{k^*},
\]
where $Q(z,a,b)(k^*)$ is the probability of absorption  {\color{jaune} in state $k^*$ at the next stage given by}
\[
 Q(z,a,b)(k^*)= \sum_{h_n,i,j,c,d} z(h_n)a(h^1_n)[i]b(h^2_n)[j]q(k_n,i,j)(k^*,c,d){\color{jaune};}
\]
{\color{jaune}$Q^0(z,a,b)$} is the probability of no absorption given by
\[
{\color{jaune}Q^0(z,a,b)}= \sum_{h_n,i,j,c,d} \sum_{k\in K^0} z(h_n)a(h^1_n)[i]b(h^2_n)[j]q(k_n,i,j)(k,c,d),
\]
and $\pi^0 \in Z_0$ is the {\color{jaune} conditional probability on not having absorbed, $i.e.$,}
\[
{\color{jaune}\forall (h_n,k,i,j,c,d) \in H_n\times K\times I \times J \times C \times D},\ \pi^0(h_n,i,j,c,d)=\frac{z(h_n)a(h^1_n)[i]b(h^2_n)[j]q(k_n,i,j)(k,c,d)}{Q^0(z,a,b)}.
\]

\item the stage payoff function $R:Z\times A\times B\to[-1,+1]$ is given by
\[
\forall (\pi_n,a,b)\in Z_0\times A \times B, \ R(\pi_n,a,b)=0,
\]
and
\[
\forall (k^*,a,b)\in K^*\times A \times B, \ R(k^*,a,b)=g(k^*).
\]
\end{itemize}

By construction, the game $\mathcal{R}$ is recursive. We denote by {\color{jaune} $\wSigma$ (resp. $\wT$) the set of behavior strategy for player 1 (resp. for player 2) } in the game $\mathcal{R}$. 

\begin{proposition}\label{aux_pass}
For every $\pi \in Z_0$ and every $n\geq 1$, the $n$-stage game $\mathcal{R}_n(\pi)$ has a value in history independent pure strategies, which is denoted by $\widetilde{v}_n(\pi)$ and
\[
\widetilde{v}_n(\pi)=v_n(\pi).
\]
Moreover, if player $2$ can uniformly defend some payoff level $v$ in the game $\mathcal{R}(\pi)$ with pure strategies then he can also uniformly defend $v$ in the game $\Gamma(\pi)$.
\end{proposition}

\begin{proof}
First, a strategy $\sigma$ of player $1$ in $\Gamma$ is a sequence of applications $(\sigma_n)_{n\geq 1}$ such that $\sigma_n$ is a mapping from $H_1^n$ to $\Delta(I)$.  By definition, this is a sequence of actions in the game $\mathcal{R}$, $i.e.$, a history independent pure strategy in $\mathcal{R}$. Similarly, a strategy $\tau$ of player $2$ in $\Gamma$ induces a sequence of actions in $\mathcal{R}$. By definition of $Q$ and $R$, it follows that for every $\pi\in Z_0$, $\sigma\in \Sigma$ and $\tau \in \T$,
\begin{align}
\gamma_n(\pi,\sigma,\tau)=\widetilde{\gamma}_n(\pi,\sigma,\tau).
\end{align}

Let $\sigma$ be an optimal strategy of player $1$ in the game $\Gamma_n(\pi)$. Consider now a pure strategy $\wtau\in\wT$. The triple $(\pi,\sigma,\wtau)$ {\color{jaune} generates a probability distribution $\PP$ on $(Z\times A \times B)^{\NN}$ such that there exists at most one play $(\pi_t,a_t,b_t)_{t\geq 1}$ that is non absorbing $\PP-a.s.$, $i.e.$,  $(\pi_t,a_t,b_t)_{t\geq 1}\in (Z_0 \times A\times B)^{\NN}$.} Define the strategy $\tau \in \T$ of player $2$ in $\Gamma(\pi)$ by setting $\tau_t=b_t$ for all $t\geq 1$. We obtain
\[
\widetilde{\gamma}_n(\pi,\sigma,\wtau)=\widetilde{\gamma}_n(\pi,\sigma,\tau) = \gamma_n(\pi,\sigma,\tau) \geq v_n(\pi).
\]
Therefore, player 1 guarantees the payoff $v_n(\pi)$ in $\mathcal{R}(\pi)$ with the history independent pure strategy $\sigma$. Similarly, player $2$ can guarantee $v_n(\pi)$ with a history independent pure strategy and $\widetilde{v}_n(\pi)=v_n(\pi)$.\\

{\color{jaune} \noindent Finally, let us assume that player $2$ can uniformly defend the payoff level $v$ with pure strategies in the game $\mathcal{R}(\pi)$. Let $\varepsilon>0$ and $\sigma\in \Sigma$. Interpreting $\sigma$ as an history-independent strategy in $\mathcal{R}$, there exist $N_0 \geq 1$ and a pure strategy $\wtau \in \wT$ such that
\begin{align}\label{audessus}
\forall n\geq N_0, \ \widetilde{\gamma}_n(\pi,\sigma,\wtau)\leq v+\varepsilon.
\end{align}

As in the previous paragraph, we can associate to the triple $(\pi,\sigma,\wtau)$ a unique play $(\pi_t,a_t,b_t)_{t\geq 1}$ in $(Z_0 \times A\times B)^{\NN}$ and define the strategy $\tau \in \T$ of player $2$ in $\Gamma(\pi)$ by setting $\tau_t=b_t$ for all $t\geq1$. We obtain
\[
\forall n\geq N_0,\ \gamma_n(\pi,\sigma,\tau)=\widetilde{\gamma}_n(\pi,\sigma,\tau)=\widetilde{\gamma}_n(\pi,\sigma,\wtau) \leq v+\varepsilon.
\]
This proves that player $2$ can uniformly defend $v$ in $\Gamma(\pi)$.
}
\end{proof}

\hspace{2mm}

We conclude by showing that the game $\mathcal{R}$ fulfills the conditions of Corollary \ref{theo3}.

\begin{proposition}\label{prop: p2-defend-Gamma} Player 2 uniformly defends $w_\infty^*\left(\Phi(\pi)\right)$ in $\Gamma(\pi)$.
\end{proposition}

\begin{proof}
We already noticed that the game $\mathcal{R}$ is recursive. Let $\pi\in \De(H^0_n\cap H'_n)\subseteq Z_0$ for some $n\geq 1$. Since player 1 is more informed than player 2 ($\pi$ supported on $H'_n$), $\pi$ can be identified as an element in $\De^1(K\times C'\times D')$ for some finite $C'$ and $D'$. By Proposition \ref{aux_pass}, we obtain that for any $\pi\in Z_0$,
\[
\widetilde{v}_n(\pi)=v_n(\pi)=\hat{v}_n\left(\Phi(\pi)\right).
\]


According to Corollary \ref{canlip}, the family $\{\hat{v}_n, n\geq 1\}$ considered as functions on $B_2$ is totally bounded, and so is the family of their linear extensions to $\De_f(B_2)$.

By Corollary \ref{theo3}, $\mathcal{R}(\pi)$ has a uniform value $w^*_\infty\left(\Phi(\pi)\right)$ in pure strategies for every $\pi\in\De^1(K\times C\times D)$. It follows from Proposition \ref{aux_pass} that player $2$ can uniformly defend $w^*_\infty\left(\Phi(\pi)\right)$ in $\Gamma(\pi)$.
\end{proof} \\ 

~~ \\ 

\bigskip 

\textbf{Acknowledgements} \ The authors thank Sylvain Sorin for his careful reading of earlier versions of this paper, whose comments have significantly improved its presentation. The authors also thank an associated editor and an anonymous referee for their numeruous helpful remarks. The authors gratefully acknowledge the support of the Agence Nationale de la Recherche, under grant ANR JEUDY, ANR-10-BLAN 0112. \\

\bibliographystyle{plain}
\bibliography{biblio_info_commune}

\end{document}